\documentclass[11pt,a4paper]{article}

\usepackage{amssymb,amsmath,amsfonts,amsthm,mathrsfs}
\usepackage{bbm}
\usepackage[marginal]{footmisc}
\usepackage{CJK}
\usepackage{mathrsfs}
\usepackage{graphicx}
\usepackage{float}
\usepackage{xcolor}
\usepackage{url}
\usepackage{hyperref}
\usepackage{relsize}
\usepackage{appendix}
\usepackage{epstopdf}

\theoremstyle{definition}

\title{\bf On Limit Formulas for Besov Seminorms and Nonlocal Perimeters in the Dunkl Setting}
\author{Huaiqian Li\footnote{Email: {\color{blue}huaiqian.li@tju.edu.cn}.  Partially supported by the National Key R\&D Program of China (Grant No. 2022YFA1006000).},\quad  Bingyao Wu\footnote{Email: {\color{blue}bingyaowu@163.com}. Partially supported by the National Key R\&D Program of China (Grant Nos. 2023YFA1010400 and 2022YFA1006003) and NNSFC  (Grant No. 12401174).}
  \vspace{2mm}
\\
{\footnotesize $^{*)}$Center for Applied Mathematics and KL-AAGDM, Tianjin University, Tianjin 300072, China}
\\
{\footnotesize $^{\dag)}$School of Mathematics and Statistics,
 Fujian Normal University, Fuzhou 350007, China}
}

\date{}

\def\R{\mathbb{R}}

\def\d{\textup{d}}

\def\1{\mathbbm{1}}

\def\<{\langle}
\def\>{\rangle}
\def\Proof.{\noindent{\bf Proof. }}

\def\newdot{{\kern.8pt\cdot\kern.8pt}}

\newtheorem{theorem}{Theorem}[section]
\newtheorem{lemma}[theorem]{Lemma}

\newtheorem{proposition}{Proposition}[section]

\newtheorem{definition}[theorem]{Definition}
\theoremstyle{definition}\newtheorem{remark}[theorem]{Remark}

\begin{document}
\allowdisplaybreaks
\maketitle
\makeatletter 
\renewcommand\theequation{\thesection.\arabic{equation}}
\@addtoreset{equation}{section}
\makeatother 

\begin{abstract}
We investigate the limiting behavior of Besov seminorms and nonlocal perimeters in Dunkl theory.
The present work generalizes two fundamental results: the Maz'ya--Shaposhnikova formula for Gagliardo seminorms and the asymptotics of (relative) fractional $s$-perimeters.
Our main contributions are twofold. First, we establish a dimension-free Maz'ya--Shaposhnikova formula via a novel, robust approach that avoids reliance on the density property of Besov spaces, offering broader applicability. Second, we prove limit formulas for nonlocal perimeters relative to bounded open sets $\Omega$, removing boundary regularity assumptions in the forward direction, while introducing a weakened regularity condition on $\partial\Omega$ (admitting fractal boundaries) for the converse, a significant improvement over existing requirements. To the best of our knowledge, the results in this second part are new even in the classic Laplacian setting.
\end{abstract}



\section{Introduction and main results}\label{sec-intro}\hskip\parindent
We consider the $n$-dimensional Euclidean space $\R^n$ endowed with the standard inner product $\langle\cdot,\cdot\rangle$ and the induced norm $|\cdot|$. Let $p\in[1,\infty)$ and $s\in(0,1)$. The fractional Sobolev space ${\rm W}^{s,p}(\R^n)$ is defined as
$${\rm W}^{s,p}(\R^n):=\{f\in {\rm L}^p(\R^n):\ [f]_{{\rm W}^{s,p}}<\infty\},$$
where ${\rm L}^p(\R^n)$ denotes the standard Lebesgue space, and $[\cdot]_{{\rm W}^{s,p}}$ is the Gagliardo seminorm given by
$$[f]_{{\rm W}^{s,p}}=\bigg(\int_{\R^n}\int_{\R^n}\frac{|f(x)-f(y)|^p}{|x-y|^{n+ps}}\,\d y\d x\bigg)^{1/p}.$$
For a comprehensive study of fractional Sobolev spaces, we refer to \cite{Leoni2023,DNPV2012}. In their seminal works \cite{BBM1,BBM2}, J. Bourgain, H. Brezis and P. Mironescu investigated the limiting behavior of the spaces ${\rm W}^{s,p}(\R^n)$ as $s\rightarrow1^-$, leading to a novel characterization of classical Sobolev spaces ${\rm W}^{1,p}(\R^n)$ and spaces of functions of bounded variation on $\R^n$. (The case when $\R^n$ is replaced by a bounded regular domain was also considered in \cite{BBM1,Davila02}). Complementing this, V. Maz'ya and T. Shaposhnikova examined the case $s\rightarrow0^+$. Among their results, they proved in \cite[Theorem 3]{MS2002} that if $f\in {\rm W}^{s_0,p}(\R^n)$ for some $s_0\in(0,1)$, then
\begin{equation}\label{MS}
\lim_{s\rightarrow0^+}s [f]_{{\rm W}^{s,p}}^p=\frac{2}{p}\omega_{n-1}\|f\|_{{\rm L}^p}^p,
\end{equation}
where $\|\cdot\|_{{\rm L}^p}$ is the standard $L^p$-norm on ${\rm L}^p(\R^n)$, and $\omega_{n-1}=\frac{2\pi^{n/2}}{\Gamma(n/2)}$ denotes the surface area of the unit sphere in $\R^n$. Here, $\Gamma$ stands for the Gamma function. We refer to \eqref{MS} as MS formula. Further developments and extensions of this formula with a dimension-dependent constant have been explored in various settings recently, see, e.g., \cite{HPXZ,DLTYY2024,CDFB2023,PSV2017,Ludwig2014,KMX2005} and references therein.

Moreover, the Gagliardo seminorm is closely related to the concept of the (fractional) $s$-perimeter, introduced by L. Caffarelli, J.-M. Roquejoffre and O. Savin in \cite{CRS2009} (with earlier connections to similar functionals studied by A. Visintin in  \cite{Vis1991}). The $s$-perimeter  has emerged as a fundamental tool in the study of $s$-minimal surfaces and phase transition problems, attracting considerable research interest since its introduction. For a comprehensive overview of recent developments, see the recent survey \cite{Serra2024} and the monograph \cite{MRT2019}, along with the references cited therein. Given a measurable set $E\subset\R^n$ and an open set $\Omega\subset\R^n$, the $s$-perimeter of $E$ relative to $\Omega$ for $s\in(0,1/2)$, is defined as
\begin{equation}\begin{split}\label{def-s-per}
{\rm Per}_s(E,\Omega)&=\int_{E\cap \Omega}\int_{E^c\cap\Omega}\frac{1}{|y-x|^{n+2s}}\,\d y\d x \cr
&\quad+ \int_{E\cap \Omega}\int_{E^c\cap\Omega^c}\frac{1}{|y-x|^{n+2s}}\,\d y\d x
+\int_{E\cap \Omega^c}\int_{E^c\cap\Omega}\frac{1}{|y-x|^{n+2s}}\,\d y\d x,
\end{split}\end{equation}
where $E^c=\R^n\setminus E$. This definition captures nonlocal interactions between points inside and outside $E$, distinguishing it from the classical (local) perimeter (see, e.g.,  \cite[Chapter 5]{EvaGar2015}). If $\1_E$ belongs to ${\rm W}^{2s,1}(\R^n)$ (e.g., $E$ is bounded with smooth enough boundary) in addition, then $E$ has finite $s$-perimeter (relative to $\R^n$), and
$${\rm Per}_s(E,\R^n)=\frac{1}{2}[\1_E]_{{\rm W}^{2s,1}}<\infty.$$
This notion appeared in \cite{BBM1,BBM2} mentioned before. Furthermore, applying the MS formula \eqref{MS} yields the asymptotic behavior:
\begin{equation}\label{asym-s-perimeter}
\lim_{s\rightarrow0^+}  s{\rm Per}_s(E,\R^n)=\frac{1}{2}\omega_{n-1}\mathscr{L}^n(E),
\end{equation}
where $\mathscr{L}^n$ denotes the $n$-dimensional Lebesgue measure.

The asymptotic formula \eqref{asym-s-perimeter} was later refined by S. Dipierro, A. Figalli, G. Palatucci and E. Valdinoci \cite{DFPV},
who established a limiting characterization for bounded domains $\Omega$ of class $C^{1,\alpha}$ for some $\alpha\in(0,1)$ and measurable sets $E$ not necessarily contained in $\Omega$, under two key assumptions: ${\rm Per}_\sigma(E,\Omega)<\infty$ for some $\sigma\in(0,1/2)$, and the existence of the limit
\begin{equation}\label{weight-density}
\iota(E):=\lim_{s\rightarrow0^+}s\int_{E\setminus B_1}|x|^{-(n+2s)}\,\d x,
\end{equation}
where $B_1$ denotes the open unit ball in $\R^n$ centered at the origin, and $\iota(E)$ captures the weighted Lebesgue measure of $E$ at infinity. Crucially, neither condition can be dropped, as explicit counterexamples provided in \cite{DFPV} demonstrate their necessity.  Recent extensions include the setting of Riemannian manifolds and ${\rm RCD}(K,\infty)$ spaces (where $K\in\R$) possessing the $L^\infty$-Liouville property \cite{CG2024}, and the $s$-fractional Gaussian perimeter framework \cite{CCLMP}. However, all such results  require $\Omega$ to satisfy standard boundary regularity conditions (typically, $C^{1,\alpha}$ or Lipschitz).

It is worth noting that, as $s\rightarrow\frac{1}{2}^-$, up to a dimension-dependent constant, the classical perimeter of a measurable set $E\subset\R^n$ can be recovered through a renormalized limit of its $s$-perimeter in various sense. For instance, this convergence holds in the sense of $\Gamma$-convergence and pointwise limits. While a detailed analysis of this topic falls outside the scope of our current work, we refer the interested reader to the recent works \cite{CDLKNP2023,Lom2019,MRT2019,BP2019,Lud2014,CaffValdV2011,AmDePhMa2011} for detailed treatments.

Both \eqref{MS} and \eqref{asym-s-perimeter} admit dimension-free formulations.  To recall these, let  $(P_t)_{t\geq0}$ be the standard heat semigroup generated by the Laplacian $\Delta$. For a bounded measurable function $f$ on $\R^n$,
$$P_tf(x):=\int_{\R^n}f(y)p_t(x,y)\,\d y,\quad t>0,\,x\in\R^n,$$
and $P_0f:=f$, where $(p_t)_{t>0}$ is the standard heat kernel given by
$$p_t(x,y)=\frac{1}{(4\pi t)^{n/2}}e^{-\frac{|x-y|^2}{4t}},\quad t>0,\,x,y\in\R^n.$$
For $p\in[1,\infty)$ and $s\in(0,1)$, we define the Besov seminorm associated with this semigroup as
$${\rm N}_{s,p}(f):=\bigg(\int_0^\infty t^{-(1+\frac{ps}{2})}\int_{\R^n}P_t(|f-f(x)|^p)(x)\,\d x\d t\bigg)^{1/p}.$$
A direct calculation employing the change-of-variables technique reveals the quantitative relationship between this seminorm and the Gagliardo seminorm:
\begin{equation}\label{Besov-fracSoblev}
{\rm N}_{s,p}(f)^p=\frac{2^{ps}\Gamma(\frac{n+ps}{2})}{\pi^{\frac{n}{2}}} [f]_{{\rm W}^{s,p}}^p,
\end{equation}
and hence, ${\rm W}^{s,p}(\R^n)=\{f\in {\rm L}^p(\R^n):\ {\rm N}_{s,p}(f)<\infty\}$. Combining \eqref{MS} with \eqref{Besov-fracSoblev} yields the following dimension-free MS formula: for every $f\in \cup_{s\in(0,1)}{\rm W}^{s,p}(\R^n)$,
\begin{equation}\label{dimfree-MS}
\lim_{s\rightarrow0^+}s{\rm N}_{s,p}(f)^p=\frac{4}{p}\|f\|_{{\rm L}^p}^p.
\end{equation}
This interpretation allows us to rewrite \eqref{asym-s-perimeter} in dimension-free form using the Besov seminorm $\mathrm{N}_{2s,1}(\cdot)$. Specifically, for $s\in(0,1/2)$, a measurable set $E\subset\mathbb{R}^n$ has finite $s$-perimeter if and only if ${\rm N}_{2s,1}(\1_E)< \infty$, and in this case,
\begin{equation}\label{dimfree-s-per}
\lim_{s\rightarrow0^+}s{\rm N}_{2s,1}(\1_E)=2\mathscr{L}^n(E).
\end{equation}

\  \   \ \  \

In this work, we aim to develop a unifying framework for these studies, based on Dunkl theory. Dunkl theory can be view as a generalization of Fourier
analysis in the Euclidean setting, and provides a far-reaching generalization of classical special functions (e.g., hypergeometric and Bessel functions) within a cohesive analytical framework. Its origins trace back to the harmonic analysis of Lie algebras and symmetric spaces in the mid-20th century, and was later shaped by fundamental contributions from C.F. Dunkl \cite{Dunkl1988,Dunkl1989,Dunkl1991,Dunkl1992}, G.J. Heckman and E.M. Opdam  \cite{HecOpd1987,Hec1987,Opdam88b,Opdam88a}, and I. Cherednik \cite{Cherednik1991,Cherednik1994}. Since its inception, Dunkl theory has grown substantially, with significant advances documented in  \cite{ADH2019,DaiFeng2016,GalYor2006,GalYor2005,ThaXu07,ThaXu05,Rosler03,Rosler1999,Rosler1998,RoslerVoit1998,Jeu} for instance. For comprehensive overviews, we refer to the surveys \cite{Anker2017,Rosler2003} and the monographs \cite{DunklXu2014,DaiXu2013}.

\subsection{Main results}\label{sec-main}\hskip\parindent
To formulate our main results, we first introduce the necessary notations and concepts; refer to Section \ref{sec-2.1} for details.

Let $(P_t^\kappa)_{t \geq 0}$ denote the Dunkl heat semigroup generated by the Dunkl Laplacian $\Delta_\kappa$. This semigroup admits the Dunkl heat kernel $(p_t^\kappa)_{t > 0}$ with respect to the weighted measure $\mu_\kappa$. For $p\in[1,\infty)$, let $\mathrm{L}^p(\mu_\kappa)$ be the standard Lebesgue space over $\mathbb{R}^n$ with respect to the measure $\mu_\kappa$, endowed with the norm $\|\cdot\|_{\mathrm{L}^p(\mu_\kappa)}$.

We now introduce Besov spaces in the context of Dunkl theory. For a more general discussion of Besov spaces and their properties, we refer to Appendix \ref{app-A}.
\begin{definition}\label{besov}
Let $p\in[1,\infty)$ and $s\in(0,\infty)$. The Besov space associated with the Dunkl heat semigroup $(P_t^\kappa)_{t\geq0}$ (or Dunkl Laplacian $\Delta_\kappa$) is defined as
$${\rm B}_{s,p}^\kappa(\R^n)=\big\{f\in {\rm L}^p(\mu_\kappa):\ {\rm N}^\kappa_{s,p}(f)<\infty\big\},$$
where ${\rm N}^\kappa_{s,p}(\cdot)$ is the Besov seminorm given by
$${\rm N}^\kappa_{s,p}(f)=\bigg(\int_0^\infty t^{-(1+\frac{sp}{2})}\int_{\R^n}P_t^\kappa(|f-f(x)|^p)(x)\,\mu_\kappa(\d x)\d t\bigg)^{1/p}.$$
\end{definition}
These spaces generalize classical Besov spaces (studied in the pioneering works of M.H. Taibleson \cite{Taibleson1964,Taibleson1965}) to the Dunkl framework.  Recent developments have further extended this framework to various operators, including the Kolmogorov--Fokker--Planck operator on $\R^n$ \cite{BGT2022,GT2020a,GT2020b}, the sub-Laplacian on Carnot groups \cite{GT2024}, and the Baouendi--Grushin operator on Grushin spaces \cite{ZWLL2024+,LiWang2024}.

Our first main result establishes a dimension-free MS-type formula.
\begin{theorem}\label{MS-dunkl}
Let $p\in[1,\infty)$. Then, for every $f\in \cup_{0<s<1}{\rm B}_{s,p}^\kappa(\R^n)$,
$$\lim_{s\rightarrow0^+}s{\rm N}^\kappa_{s,p}(f)^p=\frac{4}{p}\|f\|_{{\rm L}^p(\mu_\kappa)}^p.$$
\end{theorem}
\begin{remark}\label{remark-MS}
(i) In particular, if $\kappa\equiv0$, then the Dunkl heat semigroup $(P_t^\kappa)_{t\geq0}$ reduces to the standard heat semigroup $(P_t)_{t\geq0}$.  Consequently, Theorem \ref{MS-dunkl} recovers the classical result \eqref{MS} concerning \eqref{Besov-fracSoblev}.

(ii) Our approach is motivated by the method developed in \cite{BGT2022}, which relies on approximating functions in Besov norm through the density of Schwartz functions in the corresponding Besov space. However, in the Dunkl framework, such a density property is not generally available. To overcome this difficulty, we develop a novel technique based on approximation by simple functions at the level of the $L^p$-norm, significantly extending the methodology of \cite{BGT2022}. This approach is not only simpler but also more robust, making it applicable to broader settings beyond the Dunkl framework.
\end{remark}

Now, we turn to introduce the nonlocal perimeter in the Dunkl setting. For any $s>0$ and any pair of disjoint measurable sets $A,B\subset\R^n$, we set
$$L_s^\kappa(A,B):=\int_0^\infty t^{-(1+s)}\int_A P_t^\kappa\mathbbm{1}_B(x)\, \mu_\kappa(\d x)\d t.$$
\begin{definition}\label{rel-s-per}
Let $\Omega\subset\R^n$ be an open set and $s\in(0,1/2)$.  For a measurable set $E\subset\R^n$, the $s$-D-perimeter of $E$ relative to $\Omega$ is defined as
$${\rm Per}_s^\kappa(E,\Omega):=2\big[ L_s^\kappa(E\cap \Omega,E^c\cap \Omega)+L_s^\kappa(E\cap \Omega,E^c\cap \Omega^c)+L_s^\kappa(E\cap \Omega^c,E^c\cap \Omega)\big].$$
In particular, if $\Omega=\R^n$, we simply write ${\rm Per}_s^\kappa(E)$ instead of ${\rm Per}_s^\kappa(E,\R^n)$ and call it the $s$-D-perimeter of $E$.
\end{definition}

From Definition \ref{rel-s-per}, it is easy to see that ${\rm Per}_s^\kappa(E,\Omega)={\rm Per}_s^\kappa(E^c,\Omega)$ by the symmetry \eqref{ker-symmetry}, and if
$E\subset\Omega$ or $E^c\subset\Omega$ in addition, then ${\rm Per}_s^\kappa(E,\Omega)={\rm Per}_s^\kappa(E)$.  For more elementary properties, see Appendix \ref{app-B}.
\begin{remark}\label{rk-def}
(1) Let $\kappa\equiv0$. For disjoint measurable sets $A,B\subset\R^n$,  by a change of variables argument, we have
$$L_s^0(A,B)=\frac{2^{2s}\Gamma(\frac{n}{2}+s)}{\pi^{\frac{n}{2}}}\int_A \int_B \frac{1}{|x-y|^{n+2s}}\,\d y \d x.$$
Consequently,
$${\rm Per}_s^0(E,\Omega)=\frac{2^{2s+1}\Gamma(\frac{n}{2}+s)}{\pi^{\frac{n}{2}}}{\rm Per}_s(E,\Omega),$$
provided ${\rm Per}_s(E,\Omega)$ exists. Note that the constant $\frac{2^{2s+1}\Gamma(n/2+s)}{\pi^{n/2}}$ converges to $\frac{4}{\omega_{n-1}}$,  as $s\rightarrow0^+$.

(2)  For $0<s<1/2$, a measurable set $E\subset\R^n$ has finite $s$-D-perimeter if and only if $\mathbbm{1}_E\in {\rm B}_{2s,1}^\kappa(\R^n)$, with the identity
$${\rm Per}_s^\kappa(E)={\rm N}^\kappa_{2s,1}(\mathbbm{1}_E).$$
By Theorem \ref{MS-dunkl}, if $E\subset\R^n$ has finite $s_0$-D-perimeter for some $s_0\in (0,1/2)$, then
\begin{equation}\label{lim-per}
\lim_{s\rightarrow0^+}s{\rm Per}_s^\kappa(E)=2\mu_\kappa(E).
\end{equation}
In this case, if $\kappa\equiv0$, then \eqref{lim-per} coincides with the dimensional-free limit \eqref{dimfree-s-per}.
\end{remark}

We proceed to present our second main result, which characterizes the limiting behavior of the $s$-D-perimeter as $s\rightarrow0^+$. To state the theorem, we introduce additional notation and key concepts. Let $E\subset\R^n$ be a measurable set, and let $B_d(x,r)$ denote the ball in $\R^n$ with respect to the pseudo-metric $d$ (see Section \ref{sec-2.1}) with center $x\in\R^n$ and radius $r>0$.  We define the function $\Lambda^\kappa_E$ as follows:
\begin{equation}\label{Lambda-E}
\Lambda^\kappa_E(x,r,s)=\int_1^\infty P_t^\kappa(\mathbbm{1}_{E\cap B_d(x,r)^c})(x)\,\frac{\d t}{t^{1+s}},\quad x\in\R^n,\,r,s>0,
\end{equation}
Refer to Remark \ref{rk-s-per-lim}(4) for $\Lambda^\kappa_E$ in the particular $\kappa\equiv0$ case. From the observations in Lemma \ref{property-Lambda-E}, we see that if the limit $\lim_{s\rightarrow0^+}s\Lambda^\kappa_E(x_0,r_0,s)$ exists for some pair $(x_0,r_0)\in\R^n\times(0,\infty)$, then the limit $\lim_{s\rightarrow0^+}s\Lambda^\kappa_E(x,r,s)$ exists for all $(x,r)\in\R^n\times(0,\infty)$, is independent of both $x$ and $r$, and takes values in the interval $[0,1]$. In such cases,  we denote this limit by $\Xi^\kappa_E$:
\begin{equation}\label{Xi}
\Xi^\kappa_E:=\lim_{s\rightarrow0^+}s\Lambda^\kappa_E(x,r,s),\quad x\in\R^n,\,r>0.
\end{equation}
In particular, for $E=\R^n$, $\Xi^\kappa_{\R^n}$ always exists and equals $1$; see Remark \ref{rk-Xi-Rn}. For convenience, we say that $\Xi^\kappa_E$ exists if the above limit \eqref{Xi} exists for some pair $(x,r)\in\R^n\times(0,\infty)$.

\begin{theorem}\label{lim-rel-s-per}
Let $\Omega\subset\R^n$ be an open and bounded set, and let $E\subset\R^n$ be measurable such that ${\rm Per}_{s_0}^\kappa(E,\Omega)<\infty$ for some $s_0\in(0,1/2)$.
 Suppose $\Xi^\kappa_E$ exists. Then, the limit $\lim_{s\rightarrow0^+}s{\rm Per}_s^\kappa(E,\Omega)$ exists, and
\begin{equation}\begin{split}\label{lim-rel-s-per-1}
\lim_{s\rightarrow0^+}s{\rm Per}_s^\kappa(E,\Omega)&=2\Xi^\kappa_{E^c}\mu_\kappa(E\cap \Omega)+2\Xi^\kappa_E\mu_\kappa(E^c\cap\Omega)\\
&=2\big[(1-\Xi^\kappa_E)\mu_\kappa(E\cap \Omega)+\Xi^\kappa_E\mu_\kappa(E^c\cap\Omega)\big].
\end{split}\end{equation}
\end{theorem}

Some remarks on Theorem \ref{lim-rel-s-per} are in order.
\begin{remark}\label{rk-s-per-lim}
(1) Unlike previous works \cite{DFPV, CG2024, CCLMP}, Theorem \ref{lim-rel-s-per} does not require additional regularity (e.g., $C^{1,\alpha}$ or Lipschitz) on $\partial\Omega$. This improvement stems from our direct proof technique, which avoids reliance on Theorem \ref{MS-dunkl}.

(2) Since $\Xi^\kappa_E\in[0,1]$ by Lemma \ref{property-Lambda-E}(a), the term in the square brackets of \eqref{lim-rel-s-per-1} is a convex combination of
$\mu_\kappa(E\cap \Omega)$ and $\mu_\kappa(E^c\cap\Omega)$.

(3) If $E$ is bounded, then by Lemma \ref{lemma-lim-L}(2) and Lemma \ref{property-Lambda-E}, \eqref{lim-rel-s-per-1} simplifies to
$$\lim_{s\rightarrow0^+}s{\rm Per}_s^\kappa(E,\Omega)=2\mu_\kappa(E\cap \Omega).$$
If further $E\subset\Omega$,  this reduces to the earlier result \eqref{lim-per}.

(4)  Consider the particular case when $\kappa\equiv0$. It follows from direct calculation that $\Lambda^0_E$ admits an explicit form:
\begin{equation*}\begin{split}
\Lambda^0_E(x,r,s)&=\int_1^\infty \int_{\R^n}p_t(x,y)\mathbbm{1}_{E\setminus B(x,r)}(y)\,\d y\frac{\d t}{t^{1+s}}\\
&=4^{s}\pi^{-\frac{n}{2}} \int_{E\setminus B(x,r)}\frac{\gamma(\frac{n}{2}+s,\frac{|x-y|^2}{4})}{|x-y|^{n+2s}}\,\d y,\quad  r,s>0,\,x\in\R^n,
\end{split}\end{equation*}
where $B(x,r)=\{y\in\R^n:\ |y-x|<r\}$, and $\gamma$ is the incomplete Gamma function defined as
$$\gamma(p,u)=\int_0^u e^{-t}t^{p-1}\,\d t,\quad p>0,\,u\geq0.$$
For more properties of the incomplete Gamma function, refer to \cite{Jameson} for instance.
However, let us consider the modified function:
$$\widehat{\Lambda}_E(x,r,s):=\int_0^\infty \int_{\R^n}p_t(x,y)\mathbbm{1}_{E\setminus B(x,r)}(y)\,\d y\frac{\d t}{t^{1+s}},\quad  r,s>0,\,x\in\R^n.$$
Similar calculation leads to that
\begin{equation*}\begin{split}
\lim_{s\rightarrow0^+}s\widehat{\Lambda}_E(0,1,s)&=\lim_{s\rightarrow0^+}\frac{4^s\Gamma(\frac{n}{2}+s)}{\pi^{\frac{n}{2}}}s\int_{E\setminus B_1}\frac{1}{|y|^{n+2s}}\,\d y\\
&=\frac{\Gamma(\frac{n}{2})}{\pi^{\frac{n}{2}}}\iota(E),
\end{split}\end{equation*}
provided the limit and $\iota(E)$ (defined in \eqref{weight-density}) exist. Functions analogous to $\widehat{\Lambda}_E$, defined in terms of the heat kernel associated with the Laplacian on Riemannian manifolds and ${\rm RCD}(K,\infty)$ spaces, have been recently studied in \cite{CG2024}.
\end{remark}

We now present a converse to Theorem \ref{lim-rel-s-per} for the case where $\Omega$ is $G$-invariant, meaning that  $gx\in\Omega$ for all $g\in G,\,x\in\Omega$. This result establishes necessary and sufficient conditions for the existence of the limit $\lim_{s\rightarrow0^+}s{\rm Per}_s^\kappa(E,\Omega)$ and provides explicit formulas connecting it to the $\mu_\kappa$-measure of $E$ and its complement in $\Omega$. To ensure the result holds, we require a mild regularity condition on the boundary of $\Omega$. For this purpose, we introduce the following notation:  Given any pair of  disjoint sets $E,F\subset\R^n$, define the $r$-neighborhood of $F$ relative to $E$ as
$${\rm D}_r^E(F)=\{x\in E:\ d(x,F)\leq r\},\quad r\geq0,$$
where $ d(x,A)=\inf_{y\in A}d(x,y)$ for any $A\subset\R^n$ and any $x\in \R^n$.
\begin{theorem}\label{lim-rel-s-per-converse}
Let $\Omega\subset\R^n$ be an open, bounded, and $G$-invariant set, and let $E\subset\R^n$ be measurable such that ${\rm Per}_{s_0}^\kappa(E,\Omega)<\infty$ for some $s_0\in(0,1/2)$. Suppose that   there exist a constant $c_\ast>0$ and some $\eta>2s_0$ such that
\begin{equation}\begin{split}\label{boundary}
\mu_\kappa\big({\rm D}_r^\Omega(\Omega^c)\big)\leq c_\ast\min\{r^\eta, 1\},\quad r\in [0,1].
\end{split}\end{equation}
\begin{itemize}
\item[(a)](\emph{Balanced Measure Case}) If $\mu_\kappa(E\cap \Omega)=\mu_\kappa(E^c\cap \Omega)$, then the limit $\lim_{s\rightarrow0^+}s{\rm Per}_s^\kappa(E,\Omega)$ exists and satisfies
\begin{equation*}\label{lim-rel-s-per-2}
\lim_{s\rightarrow0^+}s{\rm Per}_s^\kappa(E,\Omega)=2\mu_\kappa(E\cap \Omega)=2\mu_\kappa(E^c\cap \Omega).
\end{equation*}

\item[(b)](\emph{Unbalanced Measure Case}) If $\mu_\kappa(E\cap \Omega)\neq\mu_\kappa(E^c\cap \Omega)$, then the limit $\lim_{s\rightarrow0^+}s{\rm Per}_s^\kappa(E,\Omega)$ exists if and only if $\Xi^\kappa_E$ exists. In this case,
    \begin{equation*}\label{lim-rel-s-per-3}
     \Xi^\kappa_E=\frac{\lim_{s\rightarrow0^+}s{\rm Per}_s^\kappa(E,\Omega)-2\mu_\kappa(E\cap \Omega)}{2[\mu_\kappa(E^c\cap \Omega)-\mu_\kappa(E\cap \Omega)]}.
    \end{equation*}
\end{itemize}
\end{theorem}

Some remarks on Theorem \ref{lim-rel-s-per-converse} are in order.
\begin{remark}
(1) In both Theorem \ref{lim-rel-s-per} and Theorem \ref{lim-rel-s-per-converse}, the assumptions regarding the finiteness of ${\rm Per}_{s_0}^\kappa(E,\Omega)$ for some $s_0\in(0,1/2)$ and the existence of the limit $\lim_{s\rightarrow0^+}s{\rm Per}_s^\kappa(E,\Omega)$ can not be removed, even in the classic case ($\kappa\equiv0$). Counterexamples demonstrating this necessity appear in \cite{DFPV}.

(2) While \eqref{boundary} assumes  $0\leq r\leq1$ for simplicity, Theorem \ref{lim-rel-s-per-converse}  still holds  when \eqref{boundary}  is satisfied for all $0\leq r\leq R$ with some $R>0$. In the special case $\kappa\equiv0$, condition \eqref{boundary} reduces to: there exist a constant $c_\ast>0$ and some $\eta>2s_0$ such that
\begin{equation}\begin{split}\label{boundary-0}
\mathscr{L}^n\big({\rm D}_r^\Omega(\Omega^c)\big)\leq c_\ast\min\{r^\eta, 1\},\quad r\in[0,1],
\end{split}\end{equation}
which is significantly weaker than the boundary regularity requirements in \cite{CG2024,CCLMP,DFPV} already mentioned in Remark \ref{rk-s-per-lim}(1).
 \begin{itemize}
\item[(i)] Condition \eqref{boundary-0} holds if the $(n-\eta)$-dimensional upper Minkowski content of $\partial\Omega$ is bounded by $c_\ast$ and $\partial\Omega$ has Minkowski (or  box-counting) dimension $n-\eta$. For instance, any bounded Lipschitz domain satisfies \eqref{boundary-0} with $\eta=1$. Let $0\leq\sigma\leq n$. For a bounded measurable set $A\subset\R^n$, recall that the upper and lower $\sigma$-dimensional  Minkowski contents of $A$ are given by
    $$\mathcal{M}^\sigma(A)=\limsup_{r\rightarrow0^+}\frac{\mathscr{L}^n(A_r)}{r^{n-\sigma}},\quad \mathcal{M}_\sigma(A)=\liminf_{r\rightarrow0^+}\frac{\mathscr{L}^n(A_r)}{r^{n-\sigma}},$$
    and the upper and lower Minkowski dimensions of $A$ are defined as
    $$\overline{{\rm dim}}_{\mathcal{M}}(A)=\sup\{\sigma\geq0:\ \mathcal{M}^\sigma(A)=\infty\},\quad \underline{{\rm dim}}_{\mathcal{M}}(A)=\sup\{\sigma\geq0:\ \mathcal{M}_\sigma(A)=\infty\},$$
    where $A_r=\{x\in\R^n:\ \inf_{y\in A}|x-y|<r\}$, namely, the open $r$-neighborhood of $A$ with respect to the Euclidean distance. Whenever
    $\overline{{\rm dim}}_{\mathcal{M}}(A)=\underline{{\rm dim}}_{\mathcal{M}}(A)=\sigma$, we say $A$ has Minkowski dimension $\sigma$. See \cite{BisPer2017,Falconer2014}.

\item[(ii)] The boundary condition \eqref{boundary-0} naturally accommodates domains with highly irregular boundaries, including fractal sets. A prototypical example is provided by the Weierstrass function:
    $$W_{a,b}(x):=\sum_{k=0}^\infty a^k \cos(2\pi b^k x),\quad x\in\R,$$
    where $0<a<1$, $b>1$ and $ab>1$. It is well known that $W_{a,b}$ is continuous everywhere but nowhere differentiable, and the graph of $W_{a,b}$ has Minkowski dimension  $2+\log_ba=:\varsigma$; see \cite{KMPY84}, \cite[Chapter 5]{BisPer2017} and \cite[Chapter 11]{Falconer2014} for more details.  Notably, recent works \cite{RenShen2021,Shen2018} resolved a key open problem, showing that when $b$ is further constrained to be an integer, the Hausdorff dimension of the graph of $W_{a,b}$ also equals to $\varsigma$. Furthermore, \cite[Theorem 3.5]{Z2005} proved that both the lower and upper $\varsigma$-dimensional Minkowski contents of the graph of $W_{a,b}$ are bounded by positive constants depending only on $a,b$. For our purposes, consider a bounded domain $\Omega$ in $\R^2$ whose boundary $\partial\Omega$ is given by the graph of $W_{a,b}$ with $b>1$ and $a\in (b^{-1}, b^{-2s_0})$. Such $\Omega$  satisfies condition \eqref{boundary-0}  with $\eta=-\log_ba\in(2s_0,1)$. Figure \ref{Weier} illustrates a domain in $\R^2$ enclosed
    by the curves $\R\ni x\mapsto W_{\frac{1}{2},3}(x)$ and $\R\ni x\mapsto f(x)=x^2-\frac{3}{2}$.
\begin{figure}[H]
  \centering
  \scalebox{0.6}[0.6]{\includegraphics{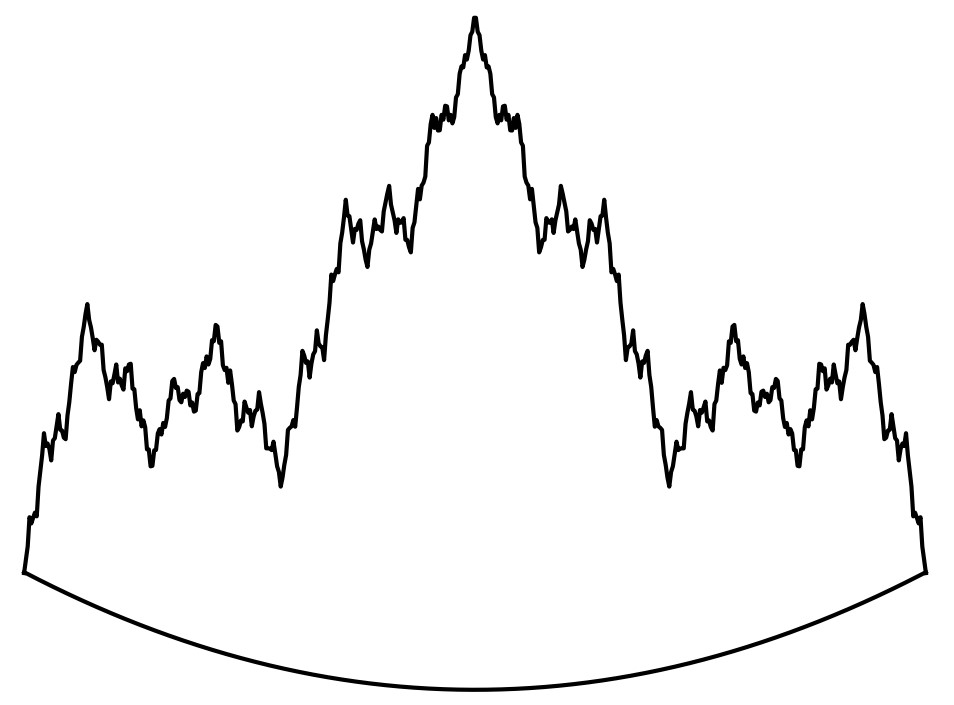}}\\
  \caption{Domain with boundary given by $W_{\frac{1}{2},3}(x)$ and $f(x)=x^2-\frac{3}{2}$.}\label{Weier}
\end{figure}
\end{itemize}

(3) The central difficulties in proving Theorem \ref{lim-rel-s-per-converse} lie in verifying the finiteness of ${\rm Per}_{s_0}^\kappa(E\cap\Omega,\Omega)$ for some $s_0\in(0,1/2)$, where both the $G$-invariance of $\Omega$ and condition \eqref{boundary} play a crucial role. The $G$-invariance requirement stems from the nontrivial interplay between the Euclidean distance and the pseudo-distance in the general setting.
\end{remark}

Motivated by  the recent work \cite{CCLMP}, we introduce a weighted version of the nonlocal perimeter in the Dunkl framework. Let us first define the key components.
 Let
$\nu_\kappa=\mathfrak{c}_\kappa^{-1}e^{-|\cdot|^2/2}\mu_\kappa$
 be the Gaussian-type weighted measure, where $\mathfrak{c}_\kappa$ is the Macdonald--Mehta constant:
\begin{equation*}\label{MM-const}
\mathfrak{c}_\kappa=\int_{\R^n}e^{-\frac{|x|^2}{2}}\,\mu_\kappa(\d x),
\end{equation*}
 and its value can be explicitly computed to be  (see \cite{Etingof,Opdam89})
 $$\mathfrak{c}_\kappa
 =(2\pi)^{\frac{n}{2}}\prod_{\beta\in\mathcal{R}_+}\frac{\Gamma(\kappa_\beta+\chi+1)}{\Gamma(\chi+1)}\in(0,\infty). $$
 Let $s\in(0,1/2)$ and let $\Omega\subset\R^n$ be an open set.  For a measurable set $E\subset\R^n$,  we define the weighted $s$-D-perimeter relative to $\Omega$ as
\begin{equation*}\label{def-w-frac-perimeter}
\widetilde{{\rm Per}}{_s^\kappa}(E,\Omega)=\widetilde{L}_s^\kappa(E\cap\Omega,E^c\cap\Omega)+ \widetilde{L}_s^\kappa(E\cap\Omega,E^c\cap\Omega^c)
+ \widetilde{L}_s^\kappa(E\cap\Omega^c,E^c\cap\Omega),
\end{equation*}
where for disjoint measurable sets $A,B\subset\R^n$,
$$\widetilde{L}_s^\kappa(A,B):=\int_A\int_B\frac{1}{|x-y|^{2\chi+n+2s}}\, \nu_\kappa(\d y) \nu_\kappa(\d x).$$
It is easy to observe that the nonlocal perimeters $\widetilde{{\rm Per}}{_s^\kappa}$ and ${\rm Per}_s^\kappa$ are not directly comparable, even when $\kappa \equiv 0$.

The result on the asymptotic behavior of $\widetilde{{\rm Per}}{_s^\kappa}(E,\Omega)$ is stated in the next proposition. Unlike the above results, our analysis does not require $\Omega$ to be bounded, connected, or have regular boundary.
\begin{proposition}\label{w-frac-perimeter}
Let $\Omega\subset\R^n$ be an open set. For every measurable subset $E$ of $\R^n$ with $\widetilde{{\rm Per}}{_{s_0}^\kappa}(E,\Omega)<\infty$ for some $s_0\in(0,1/2)$,
$$\lim_{s\rightarrow 0^+}s\widetilde{{\rm Per}}{^\kappa_s}(E,\Omega)=0.$$
\end{proposition}
\begin{proof}
Since $\widetilde{L}_{s_0}^\kappa(E\cap\Omega,E^c\cap\Omega)$ is finite by the assumption, we have
\begin{equation*}\label{w-frac-perimeter-1}\begin{split}
&\widetilde{L}_s^\kappa(E\cap\Omega,E^c\cap\Omega)\cr
&=\int_{((E\cap\Omega)\times (E^c\cap\Omega))\cap \{(x,y)\in\R^{2n}:\ |x-y|\geq1\}}\frac{1}{|x-y|^{2\chi+n+2s}}\,
\nu_\kappa\times\nu_\kappa(\d y,\d x)\cr
&\quad +\int_{((E\cap\Omega)\times (E^c\cap\Omega))\cap \{(x,y)\in\R^{2n}:\ |x-y|<1\}}\frac{1}{|x-y|^{2\chi+n+2s}}\, \nu_\kappa\times\nu_\kappa(\d y,\d x)\cr
&\leq \nu_\kappa(E\cap\Omega)\nu_\kappa(E^c\cap\Omega)+\widetilde{L}_{s_0}^\kappa(E\cap\Omega,E^c\cap\Omega)\cr
&<\infty,\quad s\in(0,s_0),
\end{split}\end{equation*}
where $\nu_\kappa\times\nu_\kappa$ stands for the product measure. This clearly leads to that
\begin{equation*}\label{w-frac-perimeter-2}\begin{split}
\lim_{s\rightarrow 0^+}s\widetilde{L}_s^\kappa(E\cap\Omega,E^c\cap\Omega)=0.
\end{split}\end{equation*}
Similarly, we also have
\begin{equation*}\label{w-frac-perimeter-2}\begin{split}
&\lim_{s\rightarrow 0^+}s\widetilde{L}_s^\kappa(E\cap\Omega,E^c\cap\Omega^c)=0,\\
&\lim_{s\rightarrow 0^+}s\widetilde{L}_s^\kappa(E\cap\Omega^c,E^c\cap\Omega)=0.
\end{split}\end{equation*}
Thus, we complete the proof.
\end{proof}

\subsection{Structure of the paper}\hskip\parindent
The paper is organized as follows. In Section \ref{sec-pre}, we recall fundamental concepts and establish necessary technical results that form the foundation for our subsequent analysis. Section \ref{sec-MS} contains the proof of Theorem \ref{MS-dunkl}, with the key innovation presented in Proposition \ref{big-time-2}. Section \ref{sec-perimeter} provides the proofs of Theorems \ref{lim-rel-s-per} and \ref{lim-rel-s-per-converse}, building on the establishment of some preparatory results. For completeness, Appendices \ref{app-A} and \ref{app-B} provide further properties of the Besov space and the nonlocal perimeter introduced in this work.

\subsection{Notation}\hskip\parindent
Throughout this work, we employ the following notation.
\begin{itemize}
\item For a set $A\subset\R^n$,
\begin{itemize}
\item[] $\1_A$  denotes its indicator function;
\item[] ${\rm diam}(A)=\sup\{|x-y|:\ x,y\in A\}$ denotes the diameter of $A$;
\item[] $A^c=\R^n\setminus A$ denotes the complement of $A$.
\end{itemize}

\item For two subsets $E,F\subset\R^n$, $E\triangle F$ denotes their symmetric difference.

\item For $k=1,2,\cdots$, $C^k(\R^n)$ denotes the space of functions on $\R^n$ with continuous derivatives up to order $k$.

\item We write $f\preceq g$ if there exists a constant $C > 0$ such that $f \leq C g$, and  $f\sim g$ if both $f\preceq g$ and $g\preceq f$ hold.

\item Positive constants are denoted by $c, C, c_1, c_2,\cdots$, and their values may vary between occurrences.
\end{itemize}

\section{Preparations}\label{sec-pre}\hskip\parindent
In this section, we begin by reviewing fundamental concepts in Dunkl theory, primarily following the expositions in \cite{ADH2019,Anker2017,Rosler2003}. Building on this foundation,  we develop several key technical tools that will be useful for later sections.

\subsection{Preliminaries on Dunkl theory}\label{sec-2.1}\hskip\parindent
Let $(\R^n, \langle\cdot,\cdot\rangle,|\cdot|)$ be the Euclidean space considered in Section \ref{sec-intro}. For each nonzero vector $\alpha\in\R^n$, we define the reflection operator  $r_\alpha: \R^n\rightarrow\R^n$ by
$$r_\alpha x =x-2\frac{\langle \alpha,x\rangle}{|\alpha|^2}\alpha,\quad x\in\R^n,$$
which represents the reflection across the hyperplane orthogonal to $\alpha$.

Let $\mathcal{R}$ be a root system in $\mathbb{R}^n$, defined as a finite, nonempty subset of $\mathbb{R}^n \setminus \{0\}$  such that  for every $\alpha\in\mathcal{R}$, $$\mathcal{R}\cap\alpha\R=\{\alpha,-\alpha\}\quad\mbox{and}\quad r_\alpha(\mathcal{R})=\mathcal{R},$$
where $\alpha\R:=\{\alpha a:\ a\in\R\}$. The reflection group (or Weyl group) $G$ generated by $\{r_\alpha:\ \alpha \in \mathcal{R}\}$ is a finite subgroup of the orthogonal group of $\mathbb{R}^n$. A positive subsystem $\mathcal{R}_+ $ is a subset of $\mathcal{R}$ such that for each root $\alpha \in \mathcal{R}$, exactly one of $\alpha$ or $-\alpha$ belongs to $\mathcal{R}_+$.

A multiplicity function $\kappa: \mathcal{R}\rightarrow[0,\infty)$ is a $G$-invariant function, meaning that $\kappa(g\alpha) = \kappa(\alpha)$ for all $g \in G$ and all $\alpha \in \mathcal{R}$. Equivalently, $\kappa$ is constant on each reflection group orbit in $\mathcal{R}$.  We mention that the $G$-invariance of $\kappa$ makes the Dunkl operators independent of the the particular choice of positive subsystem $\mathcal{R}_+ \subset \mathcal{R}$. Without loss of generality, we normalize the root system so that $|\alpha| =\sqrt{ 2}$ for all $\alpha \in \mathcal{R}$.

For $\xi \in \mathbb{R}^n$, the Dunkl operator $T_\kappa^\xi$ along $\xi$ associated with root system $\mathcal{R}$ and multiplicity function $\kappa$, initially introduced by  C.F. Dunkl  in the seminal paper \cite{Dunkl1989}, is defined by
$$T_\kappa^\xi f(x)=\langle \nabla f(x),\xi\rangle+\sum_{\alpha\in\mathcal{R}_+}\kappa(\alpha) \langle\alpha,\xi\rangle \frac{f(x)-f(r_\alpha x)}{\langle\alpha,x\rangle},\quad f\in
C^1(\R^n),\,x\in\R^n,$$
where $\nabla$ denotes the standard gradient operator. For typical examples of the Dunkl operator, we refer to \cite[Example 2.2]{Rosler2003}.  Crucially, the Dunkl operators commute, i.e., $T_\kappa^\xi\circ T_\kappa^\eta=T_\kappa^\eta\circ T_\kappa^\xi$ for all $\xi,\eta\in\R^n$.

Let $(e_l)_{l=1}^n$ be the standard orthonormal basis of $\mathbb{R}^n$. The Dunkl Laplacian, denoted $\Delta_\kappa$, is defined as $\Delta_\kappa=\sum_{l=1}^n (T_\kappa^{e_l})^2$. Explicitly, for any $f \in C^2(\mathbb{R}^n)$, the Dunkl Laplacian acts as
\begin{eqnarray*}\label{Delta}
	\Delta_\kappa f(x)=\Delta f(x)+2\sum_{\alpha\in\mathcal{R}_+}\kappa(\alpha)\bigg(\frac{\langle\alpha,\nabla f(x)\rangle}{\langle\alpha,x\rangle} - \frac{f(x)-f(r_\alpha
x)}{\langle\alpha,x\rangle^2}\bigg),\quad x\in\R^n.
\end{eqnarray*}
Clearly, due to the presence of the difference term, $\Delta_\kappa$ is a nonlocal operator.

Associated with $\Delta_\kappa$ is the weighted measure $\mu_\kappa = w_\kappa \mathscr{L}^n$, where $\mathscr{L}^n$ is the Lebesgue measure on $\mathbb{R}^n$ and the weight function $w_\kappa$ is given by
$$w_\kappa(x)=\prod_{\alpha\in\mathcal{R}_+}|\langle\alpha,x\rangle|^{2\kappa(\alpha)},\quad x\in\R^n.$$
This weight $w_\kappa$ is homogeneous of degree $2\chi$, where $\chi = \sum_{\alpha \in \mathcal{R}_+} \kappa(\alpha)$, and is $G$-invariant.  For $p\in[1,\infty]$, we denote by $\mathrm{L}^p(\mu_\kappa)$ the Lebesgue space of measurable functions on $\mathbb{R}^n$ with respect to the measure $\mu_\kappa$, equipped with the standard norm $\|\cdot\|_{\mathrm{L}^p(\mu_\kappa)}$.

Let $B(x,r)$ denote the open Euclidean ball centered at $x \in \mathbb{R}^n$ with radius $r > 0$. We write $V_\kappa(x,r) := \mu_\kappa(B(x,r))$, which satisfies
\begin{equation}\label{vol-asym}
V_\kappa(x,r)\sim r^n\prod_{\alpha\in\mathcal{R}_+}\big(|\langle\alpha,x\rangle|+r\big)^{2\kappa(\alpha)},\quad x\in\R^n,\,r>0.
\end{equation}
The measure $\mu_\kappa$ is doubling and satisfies the volume comparison property: there exists a constant $C\geq1$ such that
\begin{equation}\label{vol-comp}
C^{-1}\Big(\frac{R}{r}\Big)^n\leq\frac{V_\kappa(x,R)}{V_\kappa(x,r)}\leq C\Big(\frac{R}{r}\Big)^{n+2\chi},\quad x\in\R^n,\,0<r\leq R<\infty.
\end{equation}
While $\mu_\kappa$ is $G$-invariant, i.e., $\mu_\kappa(gA)=\mu_\kappa(A)$ for any $g\in G$ and any measurable set $A\subset\R^n$, it is neither Ahlfors regular nor translation invariant in general, where $gA:=\{gx:\ x\in A\}$.

We equip $\mathbb{R}^n$ with the $G$-invariant pseudo-metric:
$$d(x,y)=\min_{g \in G}|x-gy|,\quad x,y\in\R^n,$$
Clearly, it satisfies $d(x,y)\leq|x-y|$ for all $x,y\in\R^n$; however, the converse inequality is not true in general. Recall that, for every $x\in\R^n$ and every $r>0$,
$B_d(x,r)=\{y\in\R^n:\ d(x,y)<r\}$ is the associated pseudo-balls, which can be represented as $B_d(x,r)=\cup_{g\in G}gB(x,r)$ by the definition of the pseudo-metric $d$. Moreover,
$$V_\kappa(x,r)\leq\mu_\kappa(B_d(x,r))\leq |G|V_\kappa(x,r),\quad x\in\R^n,\,r>0,$$
where $|G|$ is the order of the reflection group $G$.

The Dunkl Laplacian $\Delta_\kappa$ is essentially self-adjoint in ${\rm L}^2(\mu_\kappa)$. Let $(P_t^\kappa)_{t\geq0}$ be the Dunkl heat semigroup generated by $\Delta_\kappa$,
i.e., for every bounded measurable function $f$ on $\R^n$, $P_0^\kappa f:=f$ and
$$P_t^\kappa f(x):=\int_{\R^n}f(y)p_t^\kappa(x,y)\,\mu_\kappa(\d y),\quad x\in\R^n,\, t>0,$$
where $(p_t^\kappa)_{t>0}$ is the Dunkl heat kernel. It is well known that $p_t^\kappa:(0,\infty)\times\R^n\times\R^n\rightarrow(0,\infty)$ is infinitely differentiable,
symmetric in $x$ and $y$, i.e.,
\begin{equation}\label{ker-symmetry}
p_t^\kappa(x,y)=p_t^\kappa(y,x),\quad x,y\in\R^n,\,t>0,
\end{equation}
 and satisfies the stochastic completeness (or conservativeness), i.e.,
\begin{equation}\label{sto-complete}
\int_{\R^n} p_t^\kappa(x,y)\,\mu_\kappa(\d y)=1,\quad x\in\R^n,\,t>0.
\end{equation}
Furthermore, the following Gaussian-type upper bound holds:
\begin{equation}\label{kernel-bound}
p_t^\kappa(x,y)\leq \frac{c_1}{\max\{V_\kappa(x,\sqrt{t}),V_\kappa(y,\sqrt{t})\}}\exp\Big(-c_2\frac{d(x,y)^2}{t}\Big),\quad t>0,\,x,y\in\R^n,
\end{equation}
for some constants $c_1,c_2>0$. It turns out that $(P_t^\kappa)_{t\geq0}$  can be extended to a strongly continuous contraction semigroup on on all ${\rm L}^p(\mu_\kappa)$ spaces ($1 \leq p < \infty$) and a contraction semigroup in ${\rm L}^\infty(\mu_\kappa)$; moreover, $(P_t^\kappa)_{t\geq0}$ is sub-Markovian, i.e., $0\leq P_t^\kappa f\leq1$ for any $t\geq0$ and any measurable function on $\R^n$ with $0\leq f\leq1$. For simplicity, we still use the same notation. Refer to \cite{ADH2019,Rosler2003,Rosler1998} for complete details and further results on the Dunkl heat semigroup/kernel.

We remark that the stochastic process associated with $(P_t^\kappa)_{t \geq 0}$ is generally a Markov jump process, but unlike L\'{e}vy processes, its increments are not necessarily stationary or independent (cf. \cite{LZ2023,GalYor2005}).  When $\kappa \equiv 0$, the Dunkl Laplacian reduces to the standard Laplacian: $\Delta_0 = \Delta$. Consequently, the semigroup and heat kernel simplify to the classical heat semigroup/kernel, respectively: $P_t^0=P_t$ and $p_t^0=p_t$ for all $t>0$.

\subsection{Useful tools}\hskip\parindent
The following regularity estimate for the Dunkl heat kernel plays a crucial role in our analysis.
While motivated by \cite[Theorem 4.1(b)]{ADH2019}, our result differs in the key aspect: it holds for all points $x,y\in\R^n$ with an additional term $ e^{|x-y|^2/(c_2t)}$.  Note that both the Euclidean metric and the pseudo-metric appear in the right hand side of \eqref{reg-kernel}.
\begin{lemma}\label{ker-lip}
There exist constants $c_1,c_2,c_3>0$ such that
\begin{equation}\label{reg-kernel}
|p_t^\kappa(x,z)-p_t^\kappa(y,z)|\leq c_1|x-y| \Big(1+\frac{|x-y|}{\sqrt{2t}}\Big)^{n+2\chi} e^{\frac{|x-y|^2}{c_2t}}
\frac{1}{\sqrt{t}V_\kappa(x,\sqrt{2t})}e^{-c_3\frac{d(x,z)^2}{ t}},
\end{equation}
for every $t>0$ and every $x,y,z\in\R^n$.
\end{lemma}
\begin{proof}
By the proof on page 2374 of \cite{ADH2019}, we can find some positive constants $c_1,c_2$ such that for any $t>0$ and any $x,y,z\in\R^n$,
\begin{equation}\begin{split}\label{pf-ker-lip-1}
|p_t^\kappa(x,z)-p_t^\kappa(y,z)|\leq c_1\frac{|x-y|}{\sqrt{t}}\int_0^1 \frac{1}{V_\kappa(z_s,\sqrt{2t})} e^{-c_2\frac{d(z,z_s)^2}{t}}\,\d s,
\end{split}\end{equation}
where $z_s:=y+s(x-y)$.

Observe that for $s\in[0,1]$ and $x,y,z\in\R^n$,
\begin{equation*}\begin{split}
|d(z,x)-d(z,z_s)|\leq|x-z_s|=(1-s)|x-y|\leq |x-y|.
\end{split}\end{equation*}
Consequently,
\begin{equation*}\begin{split}\label{pf-ker-lip-2}
|d(x,z)^2-d(z_s,z)^2|&=
|d(x,z)-d(z_s,z)|[d(x,z)+d(z,z_s)]\cr
&\leq|x-y|[2d(x,z)+|x-y|]\cr
&\leq \frac{1}{2}d(x,z)^2+3|x-y|^2,\quad s\in[0,1],\,x,y,z\in\R^n,
\end{split}\end{equation*}
where we applied the Cauchy--Schwarz inequality in the last line. This implies the lower bound
\begin{equation}\begin{split}\label{pf-ker-lip-3}
d(z_s,z)^2\geq \frac{1}{2}d(x,z)^2-3|x-y|^2,\quad s\in[0,1],\,x,y,z\in\R^n,
\end{split}\end{equation}
Furthermore, by the volume comparison property \eqref{vol-comp}, we have
\begin{equation}\begin{split}\label{pf-ker-lip-4}
\frac{V_\kappa(x,\sqrt{2t})}{V_\kappa(z_s,\sqrt{2t})}&\leq\frac{V_\kappa(z_s,|x-z_s|+\sqrt{2t})}{V_\kappa(z_s,\sqrt{2t})}
\leq\frac{V_\kappa(z_s,|x-y|+\sqrt{2t})}{V_\kappa(z_s,\sqrt{2t})}\cr
&\preceq\Big(1+\frac{|x-y|}{\sqrt{2t}}\Big)^{n+2\chi},\quad s\in[0,1],\,t>0,\,x,y\in\R^n.
\end{split}\end{equation}

Therefore, combining
\eqref{pf-ker-lip-3}, \eqref{pf-ker-lip-4} with \eqref{pf-ker-lip-1}, we immediately obtain  the desired inequality \eqref{reg-kernel}.
\end{proof}

The next result is on the ultra-contractivity of the Dunkl heat semigroup $(P_t^\kappa)_{t\geq0}$. Refer to \cite[Proposition 5.5]{Vel2020} for the particular $q=\infty$ case
of \eqref{ultra-2}.
\begin{lemma}\label{ultra-contra}
(i) Let $p\in[1,\infty]$. There exists a constant $c>0$ such that
\begin{equation}\label{ultra-1}
|P^\kappa_t f(x)|\leq c t^{-\frac{\chi+n/2}{p}}\|f\|_{{\rm L}^p(\mu_\kappa)},\quad x\in\R^n,\,t>0,\,f\in {\rm L}^p(\mu_\kappa).
\end{equation}
(ii) Let $1\leq p\leq q\leq\infty$. There exists a constant $C>0$ such that
\begin{equation}\label{ultra-2}
\|P^\kappa_t f\|_{{\rm L}^q(\mu_\kappa)}\leq  C t^{-(\frac{1}{p}-\frac{1}{q})(\chi+n/2)}\|f\|_{{\rm L}^p(\mu_\kappa)},\quad t>0,\,f\in {\rm L}^p(\mu_\kappa).
\end{equation}
\end{lemma}
\begin{proof}
(1) Let $f\in {\rm L}^1(\mu_\kappa)$. Then, by \eqref{kernel-bound} and \eqref{vol-asym}, we have
\begin{equation}\begin{split}\label{pf-ultra-L1}
&|P_t^\kappa f(x)|=\bigg|\int_{\R^n}f(y)p_t^\kappa(x,y)\,\mu_\kappa(\d y)\bigg|\cr
&\preceq \int_{\R^n}\frac{|f(y)|}{t^{\chi+n/2}}\,\mu_\kappa(\d y)= t^{-(\chi+n/2)}\|f\|_{{\rm L}^1(\mu_\kappa)},\quad t>0,\,x\in\R^n.
\end{split}\end{equation}

Let $f\in {\rm L}^\infty(\mu_\kappa)$. Then, by the stochastic completeness \eqref{sto-complete}, we have
\begin{equation}\label{pf-ultra-L0}
|P_t^\kappa f(x)|=\bigg|\int_{\R^n}f(y)p_t^\kappa(x,y)\,\mu_\kappa(\d y)\bigg|\leq\|f\|_{{\rm L}^\infty(\mu_\kappa)}, \quad t>0,\,x\in\R^n.
\end{equation}

Let $p\in(1,\infty)$ and $f\in {\rm L}^p(\mu_\kappa)$. Set $q=p/(p-1)$. By H\"{o}lder's inequality and the upper bound \eqref{kernel-bound}, we have
\begin{equation}\begin{split}\label{pf-ultra-Lp}
|P_t^\kappa f(x)|&=\bigg|\int_{\R^n}f(y)p_t^\kappa(x,y)\,\mu_\kappa(\d y)\bigg|\\
&\leq \|f\|_{{\rm L}^p(\mu_\kappa)} \|p_t^\kappa(x,\cdot)\|_{{\rm L}^q(\mu_\kappa)}\cr
&\preceq\|f\|_{{\rm L}^p(\mu_\kappa)}\bigg(\int_{\R^n}\frac{e^{-c d(x,y)^2/t}}{ V_\kappa(x,\sqrt{t})^q}\,\mu_\kappa(\d y)\bigg)^{1/q},
\quad t>0,\,x\in\R^n,
\end{split}\end{equation}
where $c$ is some positive constant. By applying  \eqref{vol-comp} and \eqref{vol-asym}, we derive that
\begin{equation}\begin{split}\label{pf-ultra-Lp+1}
&\int_{\R^n}\frac{e^{-c d(x,y)^2/t}}{ V_\kappa(x,\sqrt{t})^q}\,\mu_\kappa(\d y)\cr
&=\Big(\int_{B_d(x,\sqrt{t})}+\sum_{j=0}^\infty\int_{B_d(x,2^{j+1}\sqrt{t})\setminus B_d(x,2^j\sqrt{t})}\Big)\frac{e^{-c d(x,y)^2/t}}{
V_\kappa(x,\sqrt{t})^q}\,\mu_\kappa(\d y)\cr
&\leq\frac{\mu_{\kappa}(B_d(x,\sqrt{t}))}{V_\kappa(x,\sqrt{t})^q} + \sum_{j=0}^\infty e^{-c 4^{j}}\frac{\mu_{\kappa}(B_d(x,2^{j+1}\sqrt{t}))}{V_\kappa(x,\sqrt{t})^q}\cr
&\leq\frac{|G|}{V_\kappa(x,\sqrt{t})^{q-1}}\bigg[1+\sum_{j=0}^\infty e^{-c 4^{j}}2^{(j+1)(n+2\chi)}\bigg]\cr
&\preceq t^{-(n/2+\chi)(q-1)},\quad t>0,\,x\in\R^n,
\end{split}\end{equation}
for some constant $c>0$. Hence, \eqref{pf-ultra-Lp} and \eqref{pf-ultra-Lp+1} lead to
\begin{equation}\label{pf-ultra-Lp+2}
|P_t^\kappa f(x)|\preceq t^{-\frac{\chi+n/2}{p}}\|f\|_{{\rm L}^p(\mu_\kappa)},\quad x\in\R^n,\,t>0,\, p\in(1,\infty).
\end{equation}

Thus, putting \eqref{pf-ultra-L1}, \eqref{pf-ultra-L0} and \eqref{pf-ultra-Lp+2} together, we immediately conclude the desired inequality \eqref{ultra-1} for every
$p\in[1,\infty]$.

(2) By the contraction property of $(P_t^\kappa)_{t\geq0}$, we have
$$\| P_t^\kappa f\|_{{\rm L}^p(\mu_\kappa)}\leq\| f\|_{{\rm L}^p(\mu_\kappa)},\quad t\geq0,\,p\in[1,\infty],\,f\in {\rm L}^p(\mu_\kappa).$$
Thus, combining this with \eqref{pf-ultra-L1} (or \eqref{ultra-1} with $p=1$), by the Riesz--Thorin interpolation theorem (see e.g. \cite[Theorem 1.1.5]{Davies1989}), we
complete the proof of  \eqref{ultra-2}.
\end{proof}

Finally, we borrow a key result from \cite[Lemma 2.3]{Li2019},  whose proof follows standard techniques by applying the second inequality in \eqref{vol-comp}.
\begin{lemma}\label{int-ker-bd}
For every $\epsilon>0$, there exists a positive constant $C$ (depending on $\epsilon$ and $|G|$) such that
$$\int_{ B_d(x,r)^c}\exp\Big(-2\epsilon\frac{d(x,y)^2}{t}\Big)\, \mu_\kappa(\d y)\leq CV_\kappa(x,\sqrt{t})e^{-\epsilon r^2/t},$$
for every $r,t>0$ and every $x\in\R^n$.
\end{lemma}

\section{Dimension-free MS-type formula}\label{sec-MS}\hskip\parindent
In this section, we aim to prove Theorem \ref{MS-dunkl}. To this end, we split Theorem \ref{MS-dunkl} into two propositions in full generality.

The first one is on the short-time integration.
\begin{proposition}\label{small-time}
Let $p\in[1,\infty)$. Then, for every $f\in\cup_{s\in(0,\infty)}{\rm B}_{s,p}^\kappa(\R^n)$,
$$\lim_{s\rightarrow0^+}s\int_0^1 t^{-(1+\frac{ps}{2})}\int_{\R^n}P_t^\kappa(|f-f(x)|^p)(x)\,\mu_\kappa(\d x)\d t=0.$$
\end{proposition}
\begin{proof}
Let 
$\tau\in(0,\infty)$. Then for every $s\in(0,\tau]$ and every $f\in{\rm B}_{\tau,p}^\kappa(\R^n)$,
\begin{equation*}\begin{split}
&\int_0^1t^{-(1+\frac{ps}{2})}\int_{\R^n}P_t^\kappa(|f-f(x)|^p)(x)\,\mu_\kappa(\d x)\d t\\
&\leq \int_0^1 t^{-(1+\frac{p\tau}{2})}\int_{\R^n}P_t^\kappa(|f-f(x)|^p)(x)\,\mu_\kappa(\d x)\d t\\
&\leq {\rm N}^\kappa_{\tau,p}(f)^p<\infty.
\end{split}\end{equation*}
Thus, we finish the proof by multiplying by $s$ and taking the limit as $s\rightarrow0^+$.
\end{proof}

The second one is on the long-time integration, which is the crucial part. We emphasize that the following limit formula holds in ${\rm L}^p(\mu_\kappa)$ for all $1\leq
p<\infty$.
\begin{proposition}\label{big-time-2}
Let $p\in[1,\infty)$. Then, for every $f\in {\rm L}^p(\mu_\kappa)$,
$$\lim_{s\rightarrow0^+}s\int_1^\infty t^{-(1+\frac{ps}{2})}\int_{\R^n}P_t^\kappa(|f-f(x)|^p)(x)\,\mu_\kappa(\d x)\d t=\frac{4}{p}\|f\|_{{\rm L}^p(\mu_\kappa)}^p.$$
\end{proposition}

In order to prove Proposition \ref{big-time-2}, we consider two cases separately: $p=1$ and $p\in(1,\infty)$, which is further refined into three key lemmas.
\begin{lemma}\label{L1-upper}
For any $s>0$ and any $f\in {\rm L}^1(\mu_\kappa)$,
\begin{equation*}\begin{split}
s\int_1^\infty t^{-(1+\frac{s}{2})}\int_{\R^n}P_t^\kappa(|f-f(x)|)(x)\,\mu_\kappa(\d x)\d t
\leq 4\|f\|_{{\rm L}^1(\mu_\kappa)}.
\end{split}\end{equation*}
In particular,
\begin{equation*}\begin{split}
\limsup_{s\rightarrow0^+}s\int_1^\infty t^{-(1+\frac{s}{2})}\int_{\R^n}P_t^\kappa(|f-f(x)|)(x)\,\mu_\kappa(\d x)\d t
\leq 4\|f\|_{{\rm L}^1(\mu_\kappa)},\quad f\in {\rm L}^1(\mu_\kappa).
\end{split}\end{equation*}
\end{lemma}
\begin{proof} It is easy to observe that, for every $s>0$ and each $f\in {\rm L}^1(\mu_\kappa)$, we have
\begin{equation*}\begin{split}
&s\int_1^\infty t^{-(1+\frac{s}{2})}\int_{\R^n}P_t^\kappa(|f-f(x)|)(x)\,\mu_\kappa(\d x)\d t\\
&\leq s\int_1^\infty t^{-(1+\frac{s}{2})}\int_{\R^n}\int_{\R^n}p_t^\kappa(x,y)(|f(y)|+|f(x)|)\,\mu_\kappa(\d y)\mu_\kappa(\d x)\d t\\
&=2s\|f\|_{{\rm L}^1(\mu_\kappa)}\int_1^\infty t^{-(1+\frac{s}{2})}\,\d t=4\|f\|_{{\rm L}^1(\mu_\kappa)},
\end{split}\end{equation*}
where we applied both \eqref{sto-complete} and \eqref{ker-symmetry} in the penultimate equality. Thus, we also obtain the last assertion.
\end{proof}

\begin{lemma}\label{L1-lower}
For any $f\in {\rm L}^1(\mu_\kappa)$,
\begin{equation*}\begin{split}
\liminf_{s\rightarrow0^+}s\int_1^\infty t^{-(1+\frac{s}{2})}\int_{\R^n}P_t^\kappa(|f-f(x)|)(x)\,\mu_\kappa(\d x)\d t
\geq 4\|f\|_{{\rm L}^1(\mu_\kappa)}.
\end{split}\end{equation*}
\end{lemma}
\begin{proof} Let $\delta\in(0,1)$. Since $f\in {\rm L}^1(\mu_\kappa)$,  there exists a compact set $K_\delta\subset\R^n$ such that
\begin{equation}\label{loc-L1}
\int_{\R^n\setminus K_\delta}|f|\,\d\mu_\kappa<\delta,\quad\mbox{or equivalently},\quad \int_{ K_\delta}|f|\,\d\mu_\kappa\geq\|f\|_{{\rm L}^1(\mu_\kappa)}-\delta.
\end{equation}
For any $t>0$, we decompose the integral as follows:
\begin{align}\label{L1-lower-1}
&\int_{\R^n}P_t^\kappa(|f-f(x)|)(x)\,\mu_\kappa(\d x)=\int_{\R^n}\int_{\R^n}p_t^\kappa(x,y)|f(y)-f(x)|\,\mu_\kappa(\d y)\mu_\kappa(\d x)\cr
&=\Big(\int_{K_\delta}\int_{\R^n}+\int_{K_\delta^c}\int_{\R^n}\Big)p_t^\kappa(x,y)|f(y)-f(x)|\,\mu_\kappa(\d y)\mu_\kappa(\d x)\cr
&\geq \Big(\int_{K_\delta}\int_{K_\delta^c}+\int_{K_\delta^c}\int_{K_\delta}\Big)p_t^\kappa(x,y)|f(y)-f(x)|\,\mu_\kappa(\d y)\mu_\kappa(\d x)\cr
&\geq\int_{K_\delta}\int_{K_\delta^c}p_t^\kappa(x,y)(|f(x)|-|f(y)|)\,\mu_\kappa(\d y)\mu_\kappa(\d x)\cr
&\quad+\int_{K_\delta^c}\int_{K_\delta}p_t^\kappa(x,y)(|f(y)|-|f(x)|)\,\mu_\kappa(\d y)\mu_\kappa(\d x)\cr
&=2\int_{K_\delta}|f(x)|\int_{K_\delta^c}p_t^\kappa(x,y)\,\mu_\kappa(\d y)\mu_\kappa(\d x)
-2\int_{K_\delta}\int_{K_\delta^c}p_t^\kappa(x,y)|f(y)|\,\mu_\kappa(\d y) \mu_\kappa(\d x)\cr
&=2\bigg(\int_{K_\delta}|f|\,\d\mu_\kappa-\int_{K_\delta}|f(x)|\int_{ K_\delta}p_t^\kappa(x,y)\,\mu_\kappa(\d y)\mu_\kappa(\d x)\cr
&\quad-\int_{K_\delta^c}|f(y)|\int_{K_\delta}p_t^\kappa(x,y)\,\mu_\kappa(\d x)\mu_\kappa(\d y)\bigg)\cr
&\geq 2\bigg(\|f\|_{{\rm L}^1(\mu_\kappa)}-\delta-\int_{K_\delta}|f(x)|\int_{ K_\delta}p_t^\kappa(x,y)\,\mu_\kappa(\d y)\mu_\kappa(\d x)\cr
&\quad-\int_{K_\delta^c}|f(y)|\int_{K_\delta}p_t^\kappa(x,y)\,\mu_\kappa(\d x)\mu_\kappa(\d y)\bigg),
\end{align}
where we applied the symmetry \eqref{ker-symmetry} in the second equality, the stochastic completeness \eqref{sto-complete} and Fubini's theorem in the third equality, and
\eqref{loc-L1} in the last inequality.

By applying the ultra-contractive property in Lemma \ref{ultra-contra}(i), we obtain
\begin{equation}\begin{split}\label{L1-lower-2}
\int_{K_\delta}|f(x)|\int_{ K_\delta}p_t^\kappa(x,y)\,\mu_\kappa(\d y)\mu_\kappa(\d x)
&=\int_{K_\delta}|f(x)|P_t^\kappa\mathbbm{1}_{K_\delta}(x)\,\mu_\kappa(\d x)\\
&\leq ct^{-(\chi+\frac{n}{2})}\mu_\kappa(K_\delta)\int_{K_\delta}|f(x)|\,\mu_\kappa(\d x)\\
&\leq ct^{-(\chi+\frac{n}{2})}\mu_\kappa(K_\delta)\|f\|_{{\rm L}^1(\mu_\kappa)},\quad t>0,
\end{split}\end{equation}
for some constant $c>0$. By Fubini's theorem, \eqref{loc-L1} and the sub-Markov property $P_t^\kappa\mathbbm{1}_A\leq1$ for any measurable $A\subset\R^n$ and any $t\geq0$, it is clear that
\begin{equation}\begin{split}\label{L1-lower-3}
\int_{K_\delta^c}|f(y)|\int_{K_\delta}p_t^\kappa(x,y)\,\mu_\kappa(\d x)\mu_\kappa(\d y)\leq\int_{K_\delta^c}|f|\,\d\mu_\kappa<\delta,\quad t>0.
\end{split}\end{equation}

Thus, substituting \eqref{L1-lower-2} and \eqref{L1-lower-3} into \eqref{L1-lower-1}, we arrive at
\begin{equation*}\begin{split}
&s\int_1^\infty t^{-(1+\frac{s}{2})}\int_{\R^n}P_t^\kappa(|f-f(x)|)(x)\,\mu_\kappa(\d x)\d t\\
&\geq
2s\int_1^\infty t^{-(1+\frac{s}{2})}\left[\|f\|_{{\rm L}^1(\mu_\kappa)}-2\delta-ct^{-(\chi+\frac{n}{2})}\mu_\kappa(K_\delta)\|f\|_{{\rm L}^1(\mu_\kappa)}\right]\d t\\
&= 4 (\|f\|_{{\rm L}^1(\mu_\kappa)}-2\delta)-4c \mu_\kappa(K_\delta)\|f\|_{{\rm L}^1(\mu_\kappa)}\frac{s}{s+2\chi+n},\quad s>0,
\end{split}\end{equation*}
for some constant $c>0$. As $s\rightarrow0^+$, the last term vanishes, leaving $4 (\|f\|_{{\rm L}^1(\mu_\kappa)}-2\delta)$. Since $\delta>0$ is arbitrary, the desired result follows.
\end{proof}

The following lemma relies crucially on the standard approximation technique in ${\rm L}^p(\mu_\kappa)$ via simple functions. Let $\mathcal{F}(\R^n,\mu_\kappa)$ be the class of all simple functions on $\R^n$ that vanish outside sets with finite $\mu_\kappa$-measure. Since $\mu_\kappa$ is $\sigma$-finite, the class $\mathcal{F}(\R^n,\mu_\kappa)$ is dense in ${\rm L}^p(\mu_\kappa)$ for all $p\in(0,\infty)$. This is a well-known result in measure theory; for instance, see \cite[Theorem 1.4.13]{Grafakos2014}.
\begin{lemma}\label{p-equ}
Let $p\in(1,\infty)$. Then, for any $f\in {\rm L}^p(\mu_\kappa)$,
\begin{equation*}
\lim_{s\rightarrow0^+}s\int_1^\infty t^{-(1+\frac{ps}{2})}\int_{\R^n}P_t^\kappa(|f-f(x)|^p)(x)\,\mu_\kappa(\d x)\d t=\frac{4}{p}\|f\|_{{\rm L}^p(\mu_\kappa)}^p.
\end{equation*}
\end{lemma}
\begin{proof} Let $p\in (1,\infty)$. We divided the proof into four steps.

\textbf{\textsc{Step I}}. Let $f\in {\rm L}^p(\mu_\kappa)$. By the stochastic completeness \eqref{sto-complete} and the symmetry \eqref{ker-symmetry}, it is easy to derive that
\begin{equation}\begin{split}\label{p-equ-1}
&s\int_1^\infty t^{-(1+\frac{ps}{2})}\int_{\R^n}\int_{\R^n}p_t^\kappa(x,y)(|f(y)|^p+|f(x)|^p)\,\mu_\kappa(\d y)\mu_\kappa(\d x)\d t\\
&=2s\|f\|_{{\rm L}^p(\mu_\kappa)}^p\int_1^\infty t^{-(1+\frac{ps}{2})}\,\d t=\frac{4}{p}\|f\|_{{\rm L}^p(\mu_\kappa)}^p<\infty,\quad s>0.
\end{split}\end{equation}

\textbf{\textsc{Step II}}. Let $f\in \mathcal{F}(\R^n,\mu_\kappa)$. Applying the elementary inequality:
$$\big||a-b|^p-|a|^p-|b|^p\big|\leq c_p(|a|^{p-1}|b|+|a||b|^{p-1}),\quad a,b\in\R,$$
for some constant $c_p>0$ depending only on $p$, we have
\begin{equation*}\begin{split}
&\bigg|\int_{\R^n}P_t^\kappa(|f-f(x)|^p)(x)\,\mu_\kappa(\d x)\\
&\quad -\int_{\R^n}\int_{\R^n}p_t^\kappa(x,y)(|f(y)|^p+|f(x)|^p)\,\mu_\kappa(\d y)\mu_\kappa(\d x)\bigg|\\
&\leq \int_{\R^n}\int_{\R^n}p_t^\kappa(x,y)\Big||f(y)-f(x)|^p-|f(y)|^p-|f(x)|^p\Big|\,\mu_\kappa(\d y)\mu_\kappa(\d x)\\
&\leq c_p\int_{\R^n}\int_{\R^n}p_t^\kappa(x,y)(|f(x)|^{p-1}|f(y)|+|f(x)||f(y)|^{p-1})\,\mu_\kappa(\d y)\mu_\kappa(\d x)\\
&=2c_p\int_{\R^n}|f(x)|^{p-1}P_t^\kappa|f|(x)\,\mu_\kappa(\d x),\quad t>0,
\end{split}\end{equation*}
where the last equality is due to the symmetry \eqref{ker-symmetry} again.  Then, by H\"{o}lder's inequality and Lemma \ref{ultra-contra}(ii), we deduce
\begin{equation}\begin{split}\label{p-equ-2}
{\rm I}(f)&:=\bigg|\int_1^\infty t^{-(1+\frac{ps}{2})}\int_{\R^n}P_t^\kappa(|f-f(x)|^p)(x)\,\mu_\kappa(\d x)\d t\\
&\quad -\int_1^\infty t^{-(1+\frac{ps}{2})}\int_{\R^n}\int_{\R^n}p_t^\kappa(x,y)(|f(y)|^p+|f(x)|^p)\,\mu_\kappa(\d y)\mu_\kappa(\d x)\d t\bigg|\\
&\leq2c_p \int_1^\infty t^{-(1+\frac{ps}{2})} \int_{\R^n}|f|^{p-1}P_t^\kappa|f|\,\d\mu_\kappa \d t\\
&\leq2c_p \int_1^\infty t^{-(1+\frac{ps}{2})}\|f\|_{{\rm L}^p(\mu_\kappa)}^{p-1}\|P_t^\kappa|f|\|_{{\rm L}^p(\mu_\kappa)}\,\d t\\
&\leq \tilde{c}_p \|f\|_{{\rm L}^p(\mu_\kappa)}^{p-1}\|f\|_{{\rm L}^1(\mu_\kappa)} \int_1^\infty  t^{-[1+\frac{ps}{2}+(\chi+\frac{n}{2})(1-\frac{1}{p})]}\,\d t \\
&=\tilde{c}_p \|f\|_{{\rm L}^p(\mu_\kappa)}^{p-1}\|f\|_{{\rm L}^1(\mu_\kappa)}\frac{1}{\frac{ps}{2}+(\chi+\frac{n}{2})(1-\frac{1}{p})},\quad s>0,
\end{split}\end{equation}
for some constant $\tilde{c}_p>0$.

\textbf{\textsc{Step III}}. Let $f\in {\rm L}^p(\mu_\kappa)$. Since $\mathcal{F}(\R^n,\mu_\kappa)$ is dense in ${\rm L}^p(\mu_\kappa)$, we may take a sequence of functions $(f_m)_{m\geq1}\subset \mathcal{F}(\R^n,\mu_\kappa)$ such that $f_m\rightarrow f$ $\mu_\kappa$-a.e. as $m\rightarrow\infty$ and $|f_m|\leq |f|$ $\mu_\kappa$-a.e. for every
$m\geq1$.  Then
\begin{equation}\begin{split}\label{p-equ-3}
{\rm J}_1&:= \int_{\R^n}\int_{\R^n}p_t^\kappa(x,y)(|f(y)|^p+|f(x)|^p)\,\mu_\kappa(\d y)\mu_\kappa(\d x) \\
&\quad-\int_{\R^n}\int_{\R^n}p_t^\kappa(x,y)(|f_m(y)|^p+|f_m(x)|^p)\,\mu_\kappa(\d y)\mu_\kappa(\d x)\\
&= \int_{\R^n}\int_{\R^n}p_t^\kappa(x,y)\big[(|f(y)|^p-|f_m(y)|^p)+(|f(x)|^p-|f_m(x)|^p)\big]\,\mu_\kappa(\d y)\mu_\kappa(\d x) \\
&=2\big(\|f\|_{{\rm L}^p(\mu_\kappa)}^p-\|f_m\|_{{\rm L}^p(\mu_\kappa)}^p\big),\quad m\geq1,\,t>0.
\end{split}\end{equation}

Let $\mathcal{L}^p={\rm L}^p(\R^n\times\R^n, \mu_\kappa\times\mu_\kappa)$ be the Lebesgue space equipped with the $L^p$-norm denoted by $\|\cdot\|_{\mathcal{L}^p}$. For a function $h$ on $\R^n$, we set
$$U(h)(x,y):=p_t^\kappa(x,y)^{1/p}|h(x)-h(y)|,\quad x,y\in\R^n,\, t>0.$$
Observe that  the mapping $h\mapsto U(h)(x,y)$ is sublinear. Using the elementary inequality
$$|a^p-b^p|\leq p\max\{a^{p-1},b^{p-1}\}|a-b|,\quad a,b\geq0,$$
together with  the triangle inequality for $\|\cdot\|_{\mathcal{L}^p}$, we deduce that
\begin{equation*}\begin{split}\label{p-equ-4}
{\rm J}_2&:= \bigg| \int_{\R^n}\int_{\R^n}p_t^\kappa(x,y)|f(y)-f(x)|^p\,\mu_\kappa(\d y)\mu_\kappa(\d x) \\
&\quad-\int_{\R^n}\int_{\R^n}p_t^\kappa(x,y)|f_m(y)-f_m(x)|^p\,\mu_\kappa(\d y)\mu_\kappa(\d x)\bigg|\\
&=\big|\|U(f)\|_{\mathcal{L}^p}^p-\|U(f_m)\|_{\mathcal{L}^p}^p\big|\\
&\leq p\max\left\{\|U(f)\|_{\mathcal{L}^p}^{p-1}, \|U(f_m)\|_{\mathcal{L}^p}^{p-1}\right\}\|U(f)-U(f_m)\|_{\mathcal{L}^p},\quad m\geq1,\,t>0.
\end{split}\end{equation*}
Employing \eqref{ker-symmetry} and \eqref{sto-complete}, we obtain
\begin{equation*}\begin{split}\label{p-equ-5}
\|U(f_m)\|_{\mathcal{L}^p}^p&\leq 2^{p-1}\int_{\R^n}\int_{\R^n}p_t^\kappa(x,y)\big(|f_m(y)|^p+|f_m(x)|^p\big)\,\mu_\kappa(\d y)\mu_\kappa(\d x)\cr
&\leq2^p\|f_m\|_{{\rm L}^p(\mu_\kappa)}^p\leq2^p\|f\|_{{\rm L}^p(\mu_\kappa)}^p,\quad m\geq1,\,t>0,\\
\|U(f)\|_{\mathcal{L}^p}^p&\leq 2^{p-1}\int_{\R^n}\int_{\R^n}p_t^\kappa(x,y)\big(|f(y)|^p+|f(x)|^p\big)\,\mu_\kappa(\d y)\mu_\kappa(\d x)\cr
&\leq 2^p\|f\|_{{\rm L}^p(\mu_\kappa)}^p,\quad  t>0,
\end{split}\end{equation*}
and
\begin{equation*}\begin{split}\label{p-equ-6}
&\|U(f)-U(f_m)\|_{\mathcal{L}^p}^p\\
&\leq\int_{\R^n}\int_{\R^n}p_t^\kappa(x,y)(|f(y)-f_m(y)|+|f(x)-f_m(x)|)^p\,\mu_\kappa(\d y)\mu_\kappa(\d x)\\
&\leq 2^{p-1}\int_{\R^n}\int_{\R^n}p_t^\kappa(x,y)\big(|f(y)-f_m(y)|^p+|f(x)-f_m(x)|^p\big)\,\mu_\kappa(\d y)\mu_\kappa(\d x)\\
&=2^p\|f-f_m\|_{{\rm L}^p(\mu_\kappa)}^p,\quad m\geq1,\,t>0,
\end{split}\end{equation*}
where we additionally used the triangle inequality and the elementary fact that $(a+b)^p\leq 2^{p-1}(a^p+b^p)$ for every $p\geq1$ and every $a,b\geq0$. Hence
\begin{equation}\begin{split}\label{p-equ-7}
{\rm J}_2 &\leq p2^p\|f\|_{{\rm L}^p(\mu_\kappa)}^{p-1} \|f-f_m\|_{{\rm L}^p(\mu_\kappa)},\quad m\geq1,\,t>0.
\end{split}\end{equation}

\textbf{\textsc{Step IV}}. Let $f$ and $(f_m)_{m\geq1}$ be the same as in \textbf{\textsc{Step III}}.  Putting \eqref{p-equ-1}, \eqref{p-equ-2}, \eqref{p-equ-3} and
\eqref{p-equ-7} together, we arrive at
\begin{equation}\begin{split}\label{p-equ-8}
&\left|s\int_1^\infty t^{-(1+\frac{ps}{2})}\int_{\R^n}P_t^\kappa(|f-f(x)|^p)(x)\,\mu_\kappa(\d x)\d t-\frac{4}{p}\|f\|_{{\rm L}^p(\mu_\kappa)}^p\right|\\
&\leq s\int_1^\infty t^{-(1+\frac{ps}{2})} {\rm J}_1\,\d t + s{\rm I}(f_m) + s\int_1^\infty t^{-(1+\frac{ps}{2})} {\rm J}_2\,\d t \\
&\leq p2^{p+1}\|f\|_{{\rm L}^p(\mu_\kappa)}^{p-1} \|f-f_m\|_{{\rm L}^p(\mu_\kappa)} + \frac{4}{p}\big|\|f\|_{{\rm L}^p(\mu_\kappa)}^p-\|f_m\|_{{\rm L}^p(\mu_\kappa)}^p\big|\\
&\quad+   \tilde{c}_p \|f_m\|_{{\rm L}^p(\mu_\kappa)}^{p-1}\|f_m\|_{{\rm L}^1(\mu_\kappa)}\frac{s}{\frac{ps}{2}+(\chi+\frac{n}{2})(1-\frac{1}{p})},\quad s>0,\,m\geq1.
\end{split}\end{equation}
It is clear that, by the dominated convergence theorem, we have $\|f_m-f\|_{{\rm L}^p(\mu_\kappa)}\rightarrow0$ as $m\rightarrow\infty$.
Therefore, letting $s\rightarrow0^+$ first and then sending $m\rightarrow\infty$ in \eqref{p-equ-8}, we complete the proof of Lemma \ref{p-equ}.
\end{proof}

\begin{proof}[Proof of Proposition \ref{big-time-2}]
Proposition \ref{big-time-2} is a direct consequence of Lemma \ref{L1-upper}, Lemma \ref{L1-lower} and Lemma \ref{p-equ}.
\end{proof}

Finally, the proof of Theorem \ref{MS-dunkl} easily follows.
\begin{proof}[Proof of Theorem \ref{MS-dunkl}] Let $1\leq p<\infty$ and take $f\in \cup_{0<s<1}{\rm B}_{s,p}^\kappa(\R^n)$.
 Note that
\begin{equation*}\begin{split}
&\left|s{\rm N}^\kappa_{s,p}(f)^p-\frac{4}{p}\|f\|^p_{{\rm L}^p(\mu_\kappa)}\right|\\
&\leq s\int_0^1 t^{-(1+\frac{ps}{2})}\int_{\R^n}P_t^\kappa(|f-f(x)|^p)(x)\,\mu_\kappa(\d x)\d t \\
&\quad+ \left|s\int_1^\infty t^{-(1+\frac{ps}{2})}\int_{\R^n}P_t^\kappa(|f-f(x)|^p)(x)\,\mu_\kappa(\d x)\d t-\frac{4}{p}\|f\|_{{\rm L}^p(\mu_\kappa)}^p\right|,\quad s\in(0,1).
\end{split}\end{equation*}
Clearly, Propositions \ref{small-time} and \ref{big-time-2} imply Theorem \ref{MS-dunkl}.
\end{proof}

\section{Asymptotic behaviors of the $s$-D-perimeter}\label{sec-perimeter}\hskip\parindent
In this section, we present the proof for Theorems \ref{lim-rel-s-per} and \ref{lim-rel-s-per-converse}. To this purpose, it is better for us to prepare some preliminary results.

We need the following lemma.
\begin{lemma}\label{lemma-lim-L}
Let $E,F$ be disjoint measurable subsets of $\R^n$ such that $L_{s_0}^\kappa(E,F)<\infty$ for some $s_0\in(0,1/2)$.
\begin{itemize}
\item[(1)] If $\min\{\mu_\kappa(E),\mu_\kappa(F)\}<\infty$, then
\begin{equation}\label{lemma-lim-L-1}
\lim_{s\rightarrow0^+}s\left|L_{s}^\kappa(E,F)-\int_E\int_F\int_1^\infty p_t^\kappa(x,y)t^{-(1+s)}\,\d t \mu_\kappa(\d y) \mu_\kappa(\d x)\right|=0.
\end{equation}
\item[(2)] If $\max\{\mu_\kappa(E),\mu_\kappa(F)\}<\infty$, then
\begin{equation}\label{lemma-lim-L-2}
\lim_{s\rightarrow0^+}s L_{s}^\kappa(E,F)=0.
\end{equation}
\end{itemize}
\end{lemma}
\begin{proof} (i) Without loss of generality, suppose $\mu_\kappa(E)<\infty$. Using Fubini's theorem and the sub-Markovian property of $(P_t^\kappa)_{t\geq0}$,
we have
\begin{equation*}\begin{split}
&\int_E\int_F\int_1^\infty p_t^\kappa(x,y)t^{-(1+s)}\,\d t \mu_\kappa(\d y) \mu_\kappa(\d x)\\
&= \int_1^\infty\int_E P_t^\kappa\mathbbm{1}_F(x)t^{-(1+s)}\, \mu_\kappa(\d x)\d t\\
&\leq \mu_\kappa(E)\int_1^\infty t^{-(1+s)}\,\d t\\
& =\frac{\mu_\kappa(E)}{s}<\infty, \quad s>0.
\end{split}\end{equation*}
Now, fix $s\in(0,s_0)$. Then
\begin{equation*}\begin{split}
&\left|L_{s}^\kappa(E,F)-\int_E\int_F\int_1^\infty p_t^\kappa(x,y)t^{-(1+s)}\,\d t \mu_\kappa(\d y) \mu_\kappa(\d x)\right|\\
&=\int_E\int_F\int_0^1 p_t^\kappa(x,y)t^{-(1+s)}\,\d t \mu_\kappa(\d y) \mu_\kappa(\d x)\\
&\leq \int_E\int_F\int_0^1 p_t^\kappa(x,y)t^{-(1+s_0)}\,\d t \mu_\kappa(\d y) \mu_\kappa(\d x)\\
&\leq L_{s_0}^\kappa(E,F)<\infty.
\end{split}\end{equation*}
Thus,  multiplying by $s$ and taking the limit as $s\rightarrow0^+$, we immediately deduce that \eqref{lemma-lim-L-1} holds.

(ii) By the ultra-contractivity \eqref{ultra-1}, Fubini's theorem and the given assumption, we have
\begin{equation*}\begin{split}
&s\int_E\int_F\int_1^\infty p_t^\kappa(x,y)t^{-(1+s)}\,\d t \mu_\kappa(\d y) \mu_\kappa(\d x)\\
&=s\int_E\int_1^\infty P_t^\kappa\mathbbm{1}_F(x)t^{-(1+s)}\,\d t  \mu_\kappa(\d x)\\
&\preceq \mu_\kappa(E)\mu_\kappa(F)s\int_1^\infty t^{-(1+s+\chi+\frac{n}{2})}\,\d t \\
&=\mu_\kappa(E)\mu_\kappa(F)\frac{s}{s+\chi+\frac{n}{2}}\rightarrow0,\quad\mbox{as }s\rightarrow0^+.
\end{split}\end{equation*}
Combining this with \eqref{lemma-lim-L-1}, we conclude that \eqref{lemma-lim-L-2} holds.
\end{proof}

In the next lemma, we collect some important properties of the function $\Lambda^\kappa_E$ defined in \eqref{Lambda-E}: for a measurable set $E\subset\R^n$,
\begin{equation*}
\Lambda^\kappa_E(x,r,s)=\int_1^\infty P_t^\kappa(\mathbbm{1}_{E\setminus B_d(x,r)})(x)\,\frac{\d t}{t^{1+s}},\quad x\in\R^n,\, r,s>0.
\end{equation*}
\begin{lemma}\label{property-Lambda-E}
Let $E\subset\R^n$ be a measurable set. Suppose that the limit $\lim_{s\rightarrow0^+}s\Lambda^\kappa_E(x_\ast,r_\ast,s)$ exists for some pair $(x_\ast,r_\ast)\in\R^n\times(0,\infty)$. Then the following assertions hold:
\begin{itemize}
\item[(a)] The limit $\lim_{s\rightarrow0^+}s\Lambda^\kappa_E(x,r,s)$ exists for any $(x,r)\in\R^n\times(0,\infty)$, takes values in the interval $[0,1]$, and is independent of both $x$ and $r$. We denote this common value by $\Xi^\kappa_E$.
\item[(b)] For every $(x,r)\in\R^n\times(0,\infty)$,  $\Xi^\kappa_E=1-\Xi^\kappa_{E^c}$.
\end{itemize}
\end{lemma}
\begin{proof} We divided the proof into four parts.

(1) For any $r_\ast<R<\infty$, Lemma \ref{ultra-contra}(i) and the $G$-invariance of $\mu_\kappa$ imply
\begin{equation*}\begin{split}
|\Lambda^\kappa_E(x_\ast,r_\ast,s)-\Lambda^\kappa_E(x_\ast,R,s)|&\le\int_1^\infty\int_{B_d(x_\ast,R)\setminus B_d(x_\ast,r_\ast)}p_t^\kappa(x_\ast,y)\,\mu_\kappa(\d y)\frac{\d t}{t^{1+s}}\\
&\leq \int_1^\infty P_t^\kappa\mathbbm{1}_{B_d(x_\ast,R)}(x_\ast) t^{-(1+s)}\,\d t\\
&\preceq \mu_\kappa(B_d(x_\ast,R))\int_1^\infty t^{-(1+s+\chi+\frac{n}{2})}\,\d t\\
&\leq |G|V_\kappa(x_\ast,R)\frac{1}{s+\chi+n/2},\quad s>0,
\end{split}\end{equation*}
which yields
$$\lim_{s\rightarrow0^+}s|\Lambda^\kappa_E(x_\ast,r_\ast,s)-\Lambda^\kappa_E(x_\ast,R,s)|=0,\quad r_\ast<R<\infty.$$
Similarly, for any $0<r<r_\ast$, we have
\begin{equation*}\begin{split}
|\Lambda^\kappa_E(x_\ast,r_\ast,s)-\Lambda^\kappa_E(x_\ast,r,s)|\preceq |G|V_\kappa(x_\ast,r_\ast)\frac{1}{s+\chi+n/2},\quad s>0,
\end{split}\end{equation*}
and thus,
\begin{equation*}\begin{split}
\lim_{s\rightarrow0^+}s|\Lambda^\kappa_E(x_\ast,r_\ast,s)-\Lambda^\kappa_E(x_\ast,r,s)|=0,\quad 0<r<r_\ast.
\end{split}\end{equation*}
Consequently, the limit $\lim_{s\rightarrow0^+}s\Lambda^\kappa_E(x_\ast,r,s)$ does not depend on the choice of $r>0$.

(2) For any  $x\in\R^n$, we have
\begin{equation*}\begin{split}
&|\Lambda^\kappa_E(x,1,s)-\Lambda^\kappa_E(x_\ast,1,s)|\\
&\leq\int_1^\infty\int_{E\cap B_d(x,1)^c}|p_t^\kappa(x,z)-p_t^\kappa(x_\ast,z)|\, \mu_\kappa(\d z)t^{-(1+s)}\,\d t\\
&\quad +\int_1^\infty\int_{B_d(x,1)\triangle B_d(x_\ast,1)}p_t^\kappa(x_\ast,z)\, \mu_\kappa(\d z)t^{-(1+s)}\,\d t\\
&=: {\rm J}_1 +{\rm J}_2,\quad s>0.
\end{split}\end{equation*}
 Employing \eqref{reg-kernel}, Lemma \ref{int-ker-bd} and \eqref{vol-comp}, we deduce that
\begin{equation*}\begin{split}
{\rm J}_1 &\preceq |x-x_\ast|(1+|x-x_\ast|)^{n+2\chi}
e^{\frac{|x-x_\ast|^2}{c_1}}\\
&\quad\times\int_1^\infty \int_{ B_d(x,1)^c} \frac{1}{\sqrt{t}V_\kappa(x,\sqrt{2t})}e^{-c_2\frac{d(x,z)^2}{t}}\,  \mu_\kappa(\d z)\frac{\d t}{t^{1+s}}\\
&\preceq |x-x_\ast| (1+|x-x_\ast|)^{n+2\chi} e^{\frac{|x-x_\ast|^2}{c_1}} \int_1^\infty e^{-c_2/t}t^{-(s+3/2)}\,\d t\\
&\leq c_3|x-x_\ast| (1+|x-x_\ast|)^{n+2\chi} e^{|x-x_\ast|^2/c_1},\quad x\in\R^n,\,s>0,
\end{split}\end{equation*}
where $c_1,c_2,c_3$ are some positive constants. Applying Lemma \ref{ultra-contra}(i) again, we deduce
\begin{equation*}\begin{split}
{\rm J}_2 &\preceq \int_1^\infty P_t^\kappa\mathbbm{1}_{B_d(x,1)}(x_\ast)t^{-(1+s)}\,\d t+ \int_1^\infty P_t^\kappa\mathbbm{1}_{B_d(x_\ast,1)}(x_\ast)t^{-(1+s)}\,\d t\\
&\preceq[\mu_\kappa(B_d(x,1))+\mu_\kappa(B_d(x_\ast,1))]\int_1^\infty t^{-(1+s+\chi+\frac{n}{2})}\,\d t\\
&\preceq\frac{V_\kappa(x,1)+V_\kappa(x_\ast,1)}{s+\chi+\frac{n}{2}},\quad x\in\R^n,\,s>0.
\end{split}\end{equation*}
Combining the estimates of ${\rm J}_1$ and ${\rm J}_2$ together, we arrive at
$$\lim_{s\rightarrow0^+}s|\Lambda^\kappa_E(x,1,s)-\Lambda^\kappa_E(x_\ast,1,s)|=0,\quad x\in\R^n.$$
Thus, the limit $\lim_{s\rightarrow0^+}s\Lambda^\kappa_E(x,1,s)$ is independent of $x\in\R^n$.

(3) Using the sub-Markovian property of $(P_t^\kappa)_{t\geq0}$, it follows from  \eqref{Lambda-E} that
$$\Lambda^\kappa_E(x,r,s)\leq\int_1^\infty t^{-(1+s)}\,\d t=\frac{1}{s},\quad x\in\R^n,\,r>0,\,s>0.$$
This implies that $s\Lambda^\kappa_E(x,r,s)\in[0,1]$ for all $x\in\R^n,\,r>0,\,s>0$. Thus,  combining
(1) and (2) together, we conclude that for every $x\in\R^n$ and $r>0$, the limit
$\lim_{s\rightarrow0^+}s\Lambda^\kappa_E(x,r,s)$ exists and lies in $[0,1]$.

(4) It suffices to show that $\Xi^\kappa_{\R^n}=1$, since $\Xi^\kappa_{E^c}+\Xi^\kappa_{E}=\Xi^\kappa_{\R^n}$. Indeed, by the stochastic completeness \eqref{sto-complete},
\begin{equation*}\begin{split}
{\rm I}_1(s)&:=s\int_1^\infty\int_{\R^n}p_t^\kappa(x,y)\,\mu_\kappa(\d y)\frac{\d t}{t^{1+s}}\\
&=s\int_1^\infty t^{-(1+s)}\,\d t=1,\quad s>0,\,x\in\R^n,
\end{split}\end{equation*}
and by \eqref{ultra-1},
\begin{equation*}\begin{split}
{\rm I}_2(s)&:=s\int_1^\infty P_t^\kappa\mathbbm{1}_{B_d(x,r)}(x)\,\frac{\d t}{t^{1+s}}\\
&\preceq s\mu_\kappa(B_d(x,r))\int_1^\infty t^{-(1+s+\chi+\frac{n}{2})}\,\d t\\
&\preceq \frac{sV_\kappa(x,r)}{s+\chi+\frac{n}{2}},\quad r,s>0,\,x\in\R^n.
\end{split}\end{equation*}
Thus
\begin{equation*}\begin{split}
\Xi^\kappa_{\R^n}&=\lim_{s\rightarrow0^+}s\int_1^\infty\int_{B_d(x,r)^c}p_t^\kappa(x,y)\,\mu_\kappa(\d y)\frac{\d t}{t^{1+s}}\\
&=\lim_{s\rightarrow0^+}\big[{\rm I}_1(s)-{\rm I}_2(s)\big]=1.
\end{split}\end{equation*}
\end{proof}
\begin{remark}\label{rk-Xi-Rn}
From the proof of Lemma \ref{property-Lambda-E}(b), we observe that for every $x\in\R^n$ and any $r>0$, the limit $\lim_{s\rightarrow0^+}s\Lambda^\kappa_{\R^n}(x,r,s)$ always exists and equals $1$.
\end{remark}

Now we are ready to prove Theorems \ref{lim-rel-s-per} and \ref{lim-rel-s-per-converse}.
\begin{proof}[Proof of Theorem \ref{lim-rel-s-per}]
Note that, for every $s\in(0,1/2)$,
\begin{equation}\begin{split}\label{main2-pf-1}
{\rm Per}_s^\kappa(E,\Omega)=2\big[ L_s^\kappa(E\cap \Omega,E^c\cap \Omega)+L_s^\kappa(E\cap \Omega,E^c\cap \Omega^c)+L_s^\kappa(E\cap \Omega^c,E^c\cap \Omega)\big]
\end{split}\end{equation}
We analyze each term on the right-hand side separately.

By assumption,  $L_{s_0}^\kappa(E\cap \Omega,E^c\cap \Omega)$, $L_{s_0}^\kappa(E\cap \Omega,E^c\cap \Omega^c)$ and $L_{s_0}^\kappa(E\cap \Omega^c,E^c\cap
\Omega)$ are all finite for some $s_0\in(0,1/2)$.

(i)  We deal with $L_s^\kappa(E\cap \Omega,E^c\cap \Omega)$. Since $E\cap \Omega$ and $E^c\cap \Omega$ are clearly disjoint and both $\mu_\kappa(E\cap \Omega)$ and
$\mu_\kappa(E^c\cap \Omega)$ are finite, Lemma \ref{lemma-lim-L}(2) immediate implies
\begin{equation}\begin{split}\label{main2-pf-2}
\lim_{s\rightarrow0^+} sL_s^\kappa(E\cap \Omega,E^c\cap \Omega)=0.
\end{split}\end{equation}

(ii)  We deal with $L_s^\kappa(E\cap \Omega^c,E^c\cap \Omega)$. Since $\mu_\kappa(E^c\cap \Omega)<\infty$, Lemma \ref{lemma-lim-L}(1) gives
\begin{equation*}\label{main2-pf-ii-1}
\lim_{s\rightarrow0^+}sL_s^\kappa(E^c\cap \Omega,E\cap \Omega^c)=\lim_{s\rightarrow0^+}s\int_1^\infty\int_{E^c\cap \Omega}\int_{E\cap \Omega^c} p_t^\kappa(x,y)\,\mu_\kappa(\d
y) \mu_\kappa(\d x)\frac{\d t }{t^{1+s}}.
\end{equation*}

Fix $x_0\in  \Omega$ and $R>10{\rm diam}(\Omega)$ such that  $B_d(x_0,R)\supset\Omega$ (which is possible because $\Omega$ is bounded and the pseudo-metric $d$ is dominated by the Euclidean metric $|\cdot-\cdot|$). We split the integral:
\begin{eqnarray*}\label{main2-pf-ii-2}\begin{split}
&\lim_{s\rightarrow0^+}sL_s^\kappa(E^c\cap \Omega,E\cap \Omega^c)\\
&=\lim_{s\rightarrow0^+}s\bigg[\int_1^\infty\int_{E^c\cap \Omega}\int_{E\cap \Omega^c\cap B_d(x_0,R)} p_t^\kappa(x,y)\, \mu_\kappa(\d y) \mu_\kappa(\d x)\frac{\d t
}{t^{1+s}}\cr
&\quad + \int_1^\infty\int_{E^c\cap \Omega}\int_{E\cap  B_d(x_0,R)^c} p_t^\kappa(x,y)\,\mu_\kappa(\d y) \mu_\kappa(\d x)\frac{\d t}{t^{1+s}}\bigg].
\end{split}\end{eqnarray*}
By Lemma \ref{lemma-lim-L}(2), the first term vanishes:
\begin{equation*}\label{main2-pf-ii-3}\begin{split}
&\lim_{s\rightarrow0^+}s\int_1^\infty \int_{E^c\cap \Omega}\int_{E\cap \Omega^c\cap B_d(x_0,R)}p_t^\kappa(x,y)\, \mu_\kappa(\d y) \mu_\kappa(\d x)\frac{\d t}{t^{1+s}}\cr
&\leq\lim_{s\rightarrow0^+}sL_s^\kappa(E^c\cap \Omega,E\cap \Omega^c\cap B_d(x_0,R))=0.
\end{split}\end{equation*}
Hence
\begin{equation}\begin{split}\label{main2-pf-ii-4}
&\lim_{s\rightarrow0^+}sL_s^\kappa(E^c\cap \Omega,E\cap \Omega^c)\\
&=\lim_{s\rightarrow0^+}s\int_1^\infty\int_{E^c\cap \Omega}\int_{E\cap  B_d(x_0,R)^c} p_t^\kappa(x,y)\, \mu_\kappa(\d y) \mu_\kappa(\d x)\frac{\d t}{t^{1+s}}.
\end{split}\end{equation}

\textbf{Claim}: For $x\in E^c\cap \Omega$,
\begin{equation}\label{main2-pf-ii-5}\begin{split}
&\lim_{s\rightarrow0^+}s\int_1^\infty\int_{E\cap B_d(x_0,R)^c} p_t^\kappa(x,y)\, \mu_\kappa(\d y)\,\frac{\d t}{t^{1+s}}\cr
&=\lim_{s\rightarrow0^+}s\int_1^\infty\int_{E\cap  B_d(x,R/2)^c} p_t^\kappa(x,y)\, \mu_\kappa(\d y)\,\frac{\d t}{t^{1+s}}.
\end{split}\end{equation}
Indeed, since $B_d(x,R/2)\subset B_d(x_0,R)$, using \eqref{kernel-bound} and Lemma \ref{int-ker-bd}, we have
\begin{equation*}\label{main2-pf-ii-6}\begin{split}
&\bigg| \int_1^\infty\int_{E\setminus B_d(x_0,R)} p_t^\kappa(x,y)\,\mu_\kappa(\d y)\,\frac{\d t}{t^{1+s}}-\int_1^\infty\int_{E\setminus  B_d(x,R/2)} p_t^\kappa(x,y)\,
\mu_\kappa(\d y)\,\frac{\d t}{t^{1+s}}\bigg|\cr
&\leq \int_1^\infty\int_{B_d(x_0,R)\setminus B_d(x,R/2)}p_t^\kappa(x,y)\, \mu_\kappa(\d y)\,\frac{\d t}{t^{1+s}}\cr
&\preceq\int_1^\infty\int_{B_d(x,R/2)^c}\frac{e^{-c_1d(x,y)^2/t}}{V_\kappa(x,\sqrt{t})}\,\mu_\kappa(\d y)\frac{\d t}{t^{1+s}}\cr
&\preceq \int_1^\infty e^{-c_2R^2/t}t^{-(1+s)}\,\d t\cr
&\preceq R^{-2s},\quad s\in(0,1/2),
\end{split}\end{equation*}
for some positive constants $c_1$ and $c_2$. Multiplying the above quantities by $s$ and  taking the limit as $s\rightarrow0^+$, the claim follows.

Combining \eqref{main2-pf-ii-4}, \eqref{main2-pf-ii-5}, Lemma \ref{property-Lambda-E}, and the dominated convergence theorem, we obtain
\begin{equation}\label{main2-pf-ii-7}\begin{split}
\lim_{s\rightarrow0^+}sL_s^\kappa(E^c\cap \Omega,E\cap \Omega^c)&=\lim_{s\rightarrow0^+}\int_{E^c\cap \Omega} s\Lambda^\kappa_E(x,R/2,s)\,\mu_\kappa(\d x)\cr
&=\Xi^\kappa_E \mu_\kappa(E^c\cap \Omega).
\end{split}\end{equation}

(iii) For $L_s^\kappa(E\cap \Omega,E^c\cap \Omega^c)$, the same argument as in (ii) shows
\begin{equation}\label{main2-pf-ii-8}\begin{split}
\lim_{s\rightarrow0^+}sL_s^\kappa(E\cap \Omega,E^c\cap \Omega^c)=\Xi^\kappa_{E^c} \mu_\kappa(E\cap \Omega).
\end{split}\end{equation}
We omit the details here to save some space.

Finally, putting \eqref{main2-pf-1}, \eqref{main2-pf-2}, \eqref{main2-pf-ii-7} and \eqref{main2-pf-ii-8} together, we immediately
conclude  that $\lim_{s\rightarrow0^+}s{\rm Per}_s^\kappa(E,\Omega)$ exists and
$$\lim_{s\rightarrow0^+}s{\rm Per}_s^\kappa(E,\Omega)=2\Xi^\kappa_{E^c}\mu_\kappa(E\cap \Omega)+2\Xi^\kappa_E\mu_\kappa(E^c\cap\Omega).$$
Combining this with Lemma \ref{property-Lambda-E}, we obtain the second equality of \eqref{lim-rel-s-per-1}.
\end{proof}

\begin{proof}[Proof of Theorem \ref{lim-rel-s-per-converse}]  The proof is divided into three parts.

 \underline{\textsc{Part I}}.
We begin by verifying that ${\rm Per}_{s_0}^\kappa(E\cap\Omega,\Omega)<\infty$ for some $s_0\in(0,1/2)$. Given the assumption that ${\rm Per}_{s_0}^\kappa(E,\Omega)<\infty$ for some $s_0\in(0,1/2)$,  both $L_{s_0}^\kappa(E\cap \Omega,E^c\cap \Omega)$ and $L_{s_0}^\kappa(E\cap \Omega,E^c\cap \Omega^c)$ are finite. Thus, it suffices to prove that $L_{s_0}^\kappa(E\cap \Omega, E\cap\Omega^c)<\infty$ for some $s_0\in(0,1/2)$.

For every $r\geq0$, denote ${\rm D}_r={\rm D}_r^\Omega(\Omega^c)$ for short. Employing the upper bound \eqref{kernel-bound}, we find a constant $c_1>0$ such that
\begin{equation}\begin{split}\label{pf-thm2-2}
\int_{\Omega}\int_{\Omega^c}p_t^\kappa(x,y)\, \mu_\kappa(\d y)\mu_\kappa( \d x)&\preceq\int_{\Omega}\int_{\Omega^c}\frac{1}{V_\kappa(x,\sqrt{t})}e^{-c_1\frac{d(x,y)^2}{t}} \,\mu_\kappa(\d y) \mu_\kappa(\d x)\cr
&= \int_{\Omega\cap {\rm D}_1}\int_{\Omega^c} \frac{1}{V_\kappa(x,\sqrt{t})} e^{-c_1\frac{d(x,y)^2}{t}} \,\mu_\kappa(\d y) \mu_\kappa(\d x)\cr
&\quad+ \int_{\Omega\cap {\rm D}_1^c}\int_{\Omega^c} \frac{1}{V_\kappa(x,\sqrt{t})} e^{-c_1\frac{d(x,y)^2}{t}} \,\mu_\kappa(\d y) \mu_\kappa(\d x)\cr
&=: {\rm L}_1(t)+{\rm L}_2(t),\quad t>0,
\end{split}\end{equation}

To estimate  ${\rm L}_1(t)$, let
$$T_j=\{x\in \Omega:\ 2^{-(j+1)}< d(x,\Omega^c)\leq2^{-j}\},\quad j=0,1,2,\cdots.$$
We \textbf{claim} that
$$\cup_{j=0}^\infty T_j={\rm D}_1.$$
Indeed, the inclusion $\cup_{j=0}^\infty T_j \subset {\rm D}_1$ is immediate from the definition of $T_j$; hence, it suffices to show the converse inclusion.  Let $x\in{\rm D}_1$. Since $G$ is a finite group,  there exists some $g_x\in G$ such that
$$d(x,\Omega^c)=\inf_{y\in\Omega^c}\min_{g\in G}|gx-y|=\inf_{y\in\Omega^c}|g_x x-y|.$$
By the $G$-invariance of $\Omega$, we have $g_x x\in\Omega$. The openness of $\Omega$ guarantees the existence of $0<\delta<1$ such that $B(g_x x,\delta)\subset \Omega$, which implies $d(x,\Omega^c)\ge \delta>0$. Hence, we can choose a positive integer $k_0$ such that $2^{-(k_0+1)}<\delta$. Then, $x\in\bigcup_{j=0}^{k_0}T_j$, establishing $\bigcup_{j=0}^\infty T_j\supset {\rm D}_1$.

Note that for each $j=0,1,2,\cdots$, if $x\in T_j$ and $y\in\Omega^c$, then $d(x,y)\geq d(x,\Omega^c)> 2^{-(j+1)}$, and consequently, $y\in B_d(x,2^{-(j+1)})^c$. By Lemma \ref{int-ker-bd}  and assumption \eqref{boundary}, we deduce
\begin{equation}\begin{split}\label{pf-thm2-4}
{\rm L}_1(t)&=\sum_{j=0}^\infty\int_{T_j\cap {\rm D}_1} \int_{\Omega^c}  \frac{1}{V_\kappa(x,\sqrt{t})} e^{-c_1\frac{d(x,y)^2}{t}} \,\mu_\kappa(\d y) \mu_\kappa(\d x) \cr
&\leq\sum_{j=0}^\infty\int_{T_j\cap {\rm D}_1} \int_{B_d(x,2^{-(j+1)})^c}  \frac{1}{V_\kappa(x,\sqrt{t})} e^{-c_1\frac{d(x,y)^2}{t}}  \,\mu_\kappa(\d y) \mu_\kappa(\d x)\cr
&\preceq\sum_{j=0}^\infty e^{-c_2 4^{-j}/t} \mu_\kappa({\rm D}_{2^{-j}})\cr
&\preceq\sum_{j=0}^\infty e^{-c_2 4^{-j}/t} 2^{-\eta j},\quad t>0,
\end{split}\end{equation}
for some positive constant $c_2$.

For ${\rm L}_2(t)$, note that for any $x\in {\rm D}_1^c$ and any $y\in\Omega^c$, we have $y\in B_d(x,1)^c$.  Applying Lemma \ref{int-ker-bd} again,
\begin{equation}\begin{split}\label{pf-thm2-3}
{\rm L}_2(t)&\leq  \int_{\Omega\cap {\rm D}_1^c}\int_{\Omega^c\cap B_d(x,1)^c}  \frac{1}{V_\kappa(x,\sqrt{t})} e^{-c_1\frac{d(x,y)^2}{t}}  \mu_\kappa(\d y) \mu_\kappa(\d x)\,\cr
&\preceq \int_{\Omega} e^{-c_3/t}\,\mu_\kappa(\d x)=\mu_\kappa(\Omega) e^{-c_3/t},\quad t>0,
\end{split}\end{equation}
for some positive constant $c_3$.

Combining \eqref{pf-thm2-2}, \eqref{pf-thm2-3} and \eqref{pf-thm2-4} together, since $\eta>2s_0>0$, we arrive at
\begin{equation}\begin{split}\label{pf-thm2-5}
&L^\kappa_{s_0}(E\cap \Omega, E\cap \Omega^c)=\int_0^\infty t^{-(1+s_0)}\int_{E\cap\Omega}\int_{E\cap\Omega^c}p_t^\kappa(x,y)\,\mu_\kappa(\d y)\mu_\kappa(\d x)\d t\cr
&\preceq \int_0^\infty t^{-(1+s_0)} e^{-c_3/t}\,\d t  +    \sum_{j=0}^\infty 2^{-\eta j}\int_0^\infty t^{-(1+s_0)}e^{-c_2 4^{-j}/t}\,\d t\cr
&\sim \int_0^\infty e^{-u}u^{s_0-1}\,\d u +\sum_{j=0}^\infty 2^{-j(\eta-2s_0)}\int_0^\infty e^{-u}u^{s_0-1}\,\d u\cr
&<\infty.
\end{split}\end{equation}

 \underline{\textsc{Part II}}. In this part, we prove assertions (a) and (b). Let $x_0\in  \Omega$ and $R>20{\rm diam}(\Omega)$ such that $B_d(x_0,R)\supset\Omega$.
Following the approach in the proof of Theorem \ref{lim-rel-s-per}, we conclude that
\begin{equation*}\begin{split}
&\lim_{s\rightarrow0^+}sL_s^\kappa(E\cap \Omega, E\cap\Omega^c)\cr
&=\lim_{s\rightarrow0^+}s\int_1^\infty\int_{E\cap \Omega}\int_{E\cap  B_d(x,R/2)^c} p_t^\kappa(x,y)\, \mu_\kappa(\d y)\mu_\kappa(\d x)\,\frac{\d
t}{t^{1+s}}\cr
&=\lim_{s\rightarrow0^+}\int_{E\cap \Omega}s\Lambda^\kappa_E(x,R/2,s)\,\mu_\kappa(\d x),
\end{split}\end{equation*}
and
\begin{equation*}\begin{split}
&\lim_{s\rightarrow0^+}sL_s^\kappa(E^c\cap \Omega, E\cap\Omega^c)\cr
&=\lim_{s\rightarrow0^+}s\int_1^\infty\int_{E^c\cap \Omega}\int_{E\cap  B_d(x,R/2)^c} p_t^\kappa(x,y)\,
\mu_\kappa(\d y)\mu_\kappa(\d x)\,\frac{\d t}{t^{1+s}}\cr
&=\lim_{s\rightarrow0^+}\int_{E^c\cap \Omega}s\Lambda^\kappa_E(x,R/2,s)\,\mu_\kappa(\d x).
\end{split}\end{equation*}
Hence
\begin{equation}\begin{split}\label{pf-thm2-9}
&\lim_{s\rightarrow0^+}\frac{1}{2}[s{\rm Per}_{s}^\kappa(E,\Omega)-s{\rm Per}_{s}^\kappa(E\cap\Omega,\Omega)]\cr
&=\lim_{s\rightarrow0^+}[sL_s^\kappa(E^c\cap \Omega, E\cap\Omega^c)-sL_s^\kappa(E\cap \Omega, E\cap\Omega^c)]\cr
&=\lim_{s\rightarrow0^+}\bigg(\int_{E^c\cap \Omega}s\Lambda^\kappa_E(x,R/2,s)\,\mu_\kappa(\d x) -\int_{E\cap \Omega}s\Lambda^\kappa_E(x,R/2,s)\,\mu_\kappa(\d x)\bigg).
\end{split}\end{equation}

For brevity, we let $E_0=E\setminus B_d(0,R/2)$ and $E_x=E\setminus B_d(x,R/2)$ in what follows. Set
$$\Xi^\kappa_E(s):=s\int_1^\infty P_t^\kappa\mathbbm{1}_{E_0}(0)\,\frac{\d t}{t^{1+s}}.$$

We \textbf{claim} that for every bounded measurable subset $F$ of $\R^n$,
\begin{equation}\begin{split}\label{claim-1}
\lim_{s\rightarrow0^+}&\bigg|s\int_F\int_1^\infty P_t^\kappa\mathbbm{1}_{E_0}(x)\,\frac{\d t}{t^{1+s}}
\mu_\kappa(\d x)-  s\int_F\int_1^\infty P_t^\kappa\mathbbm{1}_{E_x}(x)\,\frac{\d t}{t^{1+s}}
\mu_\kappa(\d x)\bigg|=0,
\end{split}\end{equation}
and
\begin{equation}\label{claim-2}
\lim_{s\rightarrow0^+}\bigg|s\int_F\int_1^\infty P_t^\kappa\mathbbm{1}_{E_0}(x)\,\frac{\d t}{t^{1+s}}\mu_\kappa(\d x)  - \mu_\kappa(F)\Xi^\kappa_E(s)\bigg|=0.
\end{equation}

(i)[Proof of Theorem \ref{lim-rel-s-per-converse}(a)] Assume $\mu_\kappa(E\cap \Omega)=\mu_\kappa(E^c\cap \Omega)$. By \eqref{pf-thm2-9}, we have
\begin{equation*}\begin{split}\label{(2-a)-1}
&\lim_{s\rightarrow0^+}[s{\rm Per}_{s}^\kappa(E,\Omega)-s{\rm Per}_{s}^\kappa(E\cap\Omega,\Omega)]\cr
&= 2\lim_{s\rightarrow0^+}\bigg(s\int_{E^c\cap \Omega}\int_1^\infty P_t^\kappa\mathbbm{1}_{E_x}(x)\,\frac{\d t}{t^{1+s}}\mu_\kappa(\d x)   -  s\int_{E\cap \Omega}\int_1^\infty P_t^\kappa\mathbbm{1}_{E_x}(x)\,\frac{\d t}{t^{1+s}}\mu_\kappa(\d x)\bigg)\cr
&=2\lim_{s\rightarrow0^+}\bigg[\Big(s\int_{E^c\cap \Omega}\int_1^\infty P_t^\kappa\mathbbm{1}_{E_x}(x)\,\frac{\d t}{t^{1+s}}\mu_\kappa(\d x)    -        s\int_{E^c\cap \Omega}\int_1^\infty P_t^\kappa\mathbbm{1}_{E_0}(x)\,\frac{\d t}{t^{1+s}}\mu_\kappa(\d x)\Big)\cr
&\quad+ \Big(s\int_{E^c\cap \Omega}\int_1^\infty P_t^\kappa\mathbbm{1}_{E_0}(x)\,\frac{\d t}{t^{1+s}}\mu_\kappa(\d x)         -     s\int_{E^c\cap \Omega}\int_1^\infty P_t^\kappa\mathbbm{1}_{E_0}(0)\,\frac{\d t}{t^{1+s}}\mu_\kappa(\d x)  \Big)\cr
&\quad+\Big(s\int_{E\cap \Omega}\int_1^\infty P_t^\kappa\mathbbm{1}_{E_0}(0)\,\frac{\d t}{t^{1+s}}\mu_\kappa(\d x)   -    s\int_{E\cap \Omega}\int_1^\infty P_t^\kappa\mathbbm{1}_{E_0}(x)\,\frac{\d t}{t^{1+s}}\mu_\kappa(\d x)  \Big)\cr
&\quad+ \Big( s\int_{E\cap \Omega}\int_1^\infty P_t^\kappa\mathbbm{1}_{E_0}(x)\,\frac{\d t}{t^{1+s}}\mu_\kappa(\d x) -   s\int_{E\cap \Omega}\int_1^\infty P_t^\kappa\mathbbm{1}_{E_x}(x)\,\frac{\d t}{t^{1+s}}\mu_\kappa(\d x)\Big)\bigg].
\end{split}\end{equation*}
Since Theorem \ref{MS-dunkl} (see also \eqref{lim-per}) implies
$$\lim_{s\rightarrow0^+}s{\rm Per}_{s}^\kappa(E\cap\Omega,\Omega)=\lim_{s\rightarrow0^+}s{\rm Per}_{s}^\kappa(E\cap\Omega)=2\mu_\kappa(E\cap \Omega),$$
applying \eqref{claim-1} and \eqref{claim-2} with $F\in\{E^c\cap \Omega,E\cap \Omega\}$ yields
$$\lim_{s\rightarrow0^+}\frac{1}{2}s{\rm Per}_{s}^\kappa(E\cap\Omega,\Omega)=\mu_\kappa(E\cap \Omega),$$
which completes the proof of Theorem \ref{lim-rel-s-per-converse}(a).

(ii)[Proof of Theorem \ref{lim-rel-s-per-converse}(b)] Assume $\mu_\kappa(E\cap \Omega)\neq\mu_\kappa(E^c\cap \Omega)$ and that  $\lim_{s\rightarrow0^+}s{\rm Per}_{s}^\kappa(E,\Omega)$ exists. The sufficiency follows from Theorem \ref{lim-rel-s-per}(1). So, we only need to prove the necessity.

Applying \eqref{pf-thm2-9}, \eqref{claim-1} and \eqref{claim-2} with $F\in\{E^c\cap \Omega,E\cap \Omega\}$, and Theorem \ref{MS-dunkl}, we obtain
\begin{equation*}\begin{split}
&\lim_{s\rightarrow0^+}\Xi^\kappa_E(s)[\mu_\kappa(E^c\cap\Omega)-\mu_\kappa(E\cap\Omega)] \cr
&=\lim_{s\rightarrow0^+}\bigg( s \int_{E^c\cap \Omega}\int_1^\infty P_t^\kappa\mathbbm{1}_{E_0}(0)\,\frac{\d t}{t^{1+s}}\,\mu_\kappa(\d x)-  s \int_{E\cap \Omega}\int_1^\infty P_t^\kappa\mathbbm{1}_{E_0}(0)\,\frac{\d t}{t^{1+s}}\,\mu_\kappa(\d x)\bigg)\cr
&=\lim_{s\rightarrow0^+}\bigg\{\Big[s  \int_{E^c\cap \Omega}\int_1^\infty P_t^\kappa\mathbbm{1}_{E_0}(0)\,\frac{\d t}{t^{1+s}}\,\mu_\kappa(\d x) -   s\int_{E^c\cap \Omega}\int_1^\infty P_t^\kappa\mathbbm{1}_{E_0}(x)\,\frac{\d t}{t^{1+s}}\mu_\kappa(\d x)\Big]\cr
&\quad+\Big[s\int_{E^c\cap \Omega}\int_1^\infty P_t^\kappa\mathbbm{1}_{E_0}(x)\,\frac{\d t}{t^{1+s}}\mu_\kappa(\d x)-  s\int_{E^c\cap \Omega}\int_1^\infty P_t^\kappa\mathbbm{1}_{E_x}(x)\,\frac{\d t}{t^{1+s}}\mu_\kappa(\d x)\Big]\cr
&\quad+ \Big[ s\int_{E^c\cap \Omega}\int_1^\infty P_t^\kappa\mathbbm{1}_{E_x}(x)\,\frac{\d t}{t^{1+s}}\mu_\kappa(\d x) - s\int_{E\cap \Omega}\int_1^\infty P_t^\kappa\mathbbm{1}_{E_x}(x)\,\frac{\d t}{t^{1+s}}\mu_\kappa(\d x) \Big]\cr
&\quad+\Big[ s\int_{E\cap \Omega}\int_1^\infty P_t^\kappa\mathbbm{1}_{E_x}(x)\,\frac{\d t}{t^{1+s}}\mu_\kappa(\d x)  -   s \int_{E\cap \Omega}\int_1^\infty P_t^\kappa\mathbbm{1}_{E_0}(x)\,\frac{\d t}{t^{1+s}}\,\mu_\kappa(\d x) \Big]\cr
&\quad +  \Big[s \int_{E\cap \Omega}\int_1^\infty P_t^\kappa\mathbbm{1}_{E_0}(x)\,\frac{\d t}{t^{1+s}}\,\mu_\kappa(\d x)  -  s \int_{E\cap \Omega}\int_1^\infty P_t^\kappa\mathbbm{1}_{E_0}(0)\,\frac{\d t}{t^{1+s}}\,\mu_\kappa(\d x)\Big]\bigg\}\cr
&=\lim_{s\rightarrow0^+}\frac{1}{2}[s{\rm Per}_{s}^\kappa(E,\Omega)-s{\rm Per}_{s}^\kappa(E\cap\Omega,\Omega)]\cr
&=\lim_{s\rightarrow0^+}\frac{1}{2} s{\rm Per}_{s}^\kappa(E,\Omega)-\mu_\kappa(E\cap\Omega),
\end{split}\end{equation*}
which together with Lemma \ref{property-Lambda-E} implies that Theorem \ref{lim-rel-s-per-converse}(b) holds.

\underline{\textsc{Part III}}.
Here,  we aim to prove the last \textbf{claim}, namely, \eqref{claim-1} and \eqref{claim-2}.

We begin with the proof of \eqref{claim-1}. Applying Lemma \ref{ultra-contra}(i) yields the estimate:
\begin{equation*}\begin{split}
&\bigg|s\int_F\int_1^\infty P_t^\kappa\mathbbm{1}_{E_0}(x)\,\frac{\d t}{t^{1+s}}\mu_\kappa(\d x)-  s\int_F\int_1^\infty P_t^\kappa\mathbbm{1}_{E_x}(x)\,\frac{\d t}{t^{1+s}}\mu_\kappa(\d x)\bigg|\cr
&\leq s\int_1^\infty\int_F[P_t^\kappa\mathbbm{1}_{B_d(0,R/2)}(x)+P_t^\kappa\mathbbm{1}_{B_d(x,R/2)}(x)]\,\mu_\kappa(\d x)\frac{\d t}{t^{1+s}}\cr
&\preceq s \mu_\kappa(F) [\mu_\kappa(B_d(0,R/2)) + \mu_\kappa(B_d(x,R/2))] \int_1^\infty t^{-\frac{n}{2}-s-1-\chi}\,\d t\cr
&=\mu_\kappa(F) [\mu_\kappa(B_d(0,R/2)) + \mu_\kappa(B_d(x,R/2))] \frac{s}{s+n/2+\chi},
\end{split}\end{equation*}
which clearly vanishes as $s\rightarrow0^+$.

For \eqref{claim-2}, fix $r_0>0$ sufficiently large so that $B(0,r_0)\supset F$. Combining Lemma \ref{int-ker-bd} and Lemma \ref{ker-lip}, we obtain
\begin{equation*}\begin{split}\label{(2-b)-1}
&\bigg|s\int_F\int_1^\infty P_t^\kappa\mathbbm{1}_{E_0}(x)\,\frac{\d t}{t^{1+s}}\mu_\kappa(\d x)-  \mu_\kappa(F)\Xi^\kappa_E(s) \bigg|\cr
&\leq s \int_1^\infty\int_F\int_{E_0}|p_t^\kappa(x,z)-p_t^\kappa(0,z)|\,\mu_\kappa(\d z)\mu_\kappa(\d x)\frac{\d t}{t^{1+s}}\cr
&\preceq s\int_1^\infty\int_{B(0,r_0)}\int_{B_d(0,R/2)^c} |x|\Big(1+\frac{|x|}{\sqrt{t}}\Big)^{n+2\chi}e^{\frac{|x|^2}{c_4 t}}\cr
 &\quad\times \frac{1}{\sqrt{t}V_\kappa(x,\sqrt{t})} e^{-c_5\frac{d(x,z)^2}{t}}\,\mu_\kappa(\d z)\mu_\kappa(\d x)\frac{\d t}{t^{1+s}} \cr
&\preceq s r_0(1+r_0)^{n+2\chi}e^{r_0^2/c_4}\int_1^\infty t^{-(s+1)}e^{-c_6R^2/t}\,\d t\cr
&\preceq sr_0(1+r_0)^{n+2\chi}e^{r_0^2/c_4}R^{-2s},\quad  s>0,
\end{split}\end{equation*}
where $c_4,c_5,c_6$ are positive constants. The limit in \eqref{claim-2} follows immediately from this estimate.

Therefore, the proof is completed.
\end{proof}

\subsection*{Acknowledgment}\hskip\parindent
The authors gratefully acknowledge Prof. Yuan Liu (CAS), Prof. Yong Han (SZU), and Prof. Huonan Lin (FJNU) for their insightful discussions and valuable clarifications on fractal theory.

\begin{appendix}
\section{Appendix}\label{app-A}\hskip\parindent
In this appendix, building upon recent work on Besov spaces in Grushin spaces \cite{ZWLL2024+}, we introduce analogue Besov spaces associated with the Dunkl Laplacian, generalizing those defined in Definition \ref{besov}, and establish some of their properties.
\begin{definition}\label{general-besov}
Let $p\in[1,\infty)$, $q\in[1,\infty]$ and $s\in(0,\infty)$. We define
$${\rm B}_{s,p}^{\kappa,q}(\R^n):=\big\{f\in {\rm L}^p(\mu_\kappa):\ {\rm N}^{\kappa,s}_{p,q}(f)<\infty\big\},$$
where the Besov seminorm is given by
\begin{equation*}
{\rm N}^{\kappa,q}_{s,p}(f)=\begin{cases}
\Big(\int_0^\infty t^{-(1+\frac{sq}{2})}\Big(\int_{\R^n}P_t^\kappa(|f-f(x)|^p)(x)\,\mu_\kappa(\d x)\Big)^{q/p}\d t\Big)^{1/q},\quad&{\mbox{if }q\neq\infty},\\

\\

\sup_{t>0} t^{-s/2}\Big(\int_{\R^n}P_t^\kappa(|f-f(x)|^p)(x)\,\mu_\kappa(\d x)\Big)^{1/p},\quad&{\mbox{if }q=\infty}.
\end{cases}
\end{equation*}
\end{definition}

It is clear that ${\rm B}_{s,p}^{\kappa,p}(\R^n)={\rm B}_{s,p}^\kappa(\R^n)$ for all $(p,s)\in[1,\infty)\times(0,\infty)$.  We give the following elementary remark.
\begin{remark}\label{rem-A}
(1) Let $p,q\in[1,\infty)$ and $s\in(0,\infty)$. For any measurable function $f$ on $\R^n$, if ${\rm N}^{\kappa,q}_{s,p}(f)<\infty$, then
$$\widetilde{{\rm N}}^{\kappa,q}_{s,p}(f):=\Big(\int_0^1 t^{-(1+\frac{sq}{2})}\Big(\int_{\R^n}P_t^\kappa(|f-f(x)|^p)(x)\,\mu_\kappa(\d x)\Big)^{q/p}\d t\Big)^{1/q}<\infty.$$
Conversely, for any $f\in {\rm L}^p(\mu_\kappa)$, if $\widetilde{{\rm N}}^{\kappa,q}_{s,p}(f)<\infty$, then ${\rm N}^{\kappa,q}_{s,p}(f)<\infty$. Indeed, by the elementary inequality $(a+b)^p\leq 2^{p-1}(a^p+b^p)$ for all $a,b\geq0$, Fubini's theorem, conservativeness and symmetry of the Dunkl heat kernel, we have
\begin{equation*}\begin{split}
&{\rm N}^{\kappa,q}_{s,p}(f)^q=\widetilde{{\rm N}}^{\kappa,q}_{s,p}(f)^q +
\int_1^\infty t^{-(1+\frac{sq}{2})}\Big(\int_{\R^n}\int_{\R^n}p_t^\kappa(x,y)|f(y)-f(x)|^p\,\mu_\kappa(\d y)\mu_\kappa(\d x)\Big)^{q/p}\d t\\
&\leq \widetilde{{\rm N}}^{\kappa,q}_{s,p}(f)^q +
\int_1^\infty t^{-(1+\frac{sq}{2})}\Big(2^{p-1}\int_{\R^n}\int_{\R^n}p_t^\kappa(x,y)(|f(y)|^p + |f(x)|^p) \,\mu_\kappa(\d x)\mu_\kappa(\d x)\Big)^{q/p}\d t\\
&=\widetilde{{\rm N}}^{\kappa,q}_{s,p}(f)^q + 2^q\|f\|_{{\rm L}^p(\mu_\kappa)}^q \int_1^\infty t^{-(1+\frac{sq}{2})}\,\d t\\
&=\widetilde{{\rm N}}^{\kappa,q}_{s,p}(f)^q +\frac{2^{q+1}}{sq}\|f\|_{{\rm L}^p(\mu_\kappa)}^q<\infty.
\end{split}\end{equation*}

(2) Let $p\in[1,\infty)$, $q=\infty$ and $s\in(0,\infty)$. For any measurable function $f$ on $\R^n$, if ${\rm N}^{\kappa,\infty}_{s,p}(f)<\infty$, then
$$\widetilde{{\rm N}}^{\kappa,\infty}_{s,p}(f):=\limsup_{t\rightarrow0^+} t^{-s/2}\Big(\int_{\R^n}P_t^\kappa(|f-f(x)|^p)(x)\,\mu_\kappa(\d x)\Big)^{1/p}<\infty.$$
Conversely, for any $f\in {\rm L}^p(\mu_\kappa)$, if $\widetilde{{\rm N}}^{\kappa,\infty}_{s,p}(f)<\infty$, then ${\rm N}^{\kappa,\infty}_{s,p}(f)<\infty$. In fact, the finiteness of $\widetilde{{\rm N}}^{\kappa,\infty}_{s,p}(f)$ implies that there exists some $\delta>0$ such that
\begin{equation*}
\sup_{t\in(0,\delta)}t^{-s/2}\Big(\int_{\R^n}P_t^\kappa(|f-f(x)|^p)(x)\,\mu_\kappa(\d x)\Big)^{1/p}<\infty,
\end{equation*}
and similar as the argument in (1),  we derive
\begin{equation*}\begin{split}
\sup_{t\geq\delta}t^{-s/2}\Big(\int_{\R^n}P_t^\kappa(|f-f(x)|^p)(x)\,\mu_\kappa(\d x)\Big)^{1/p}\leq 2\delta^{-s/2}\|f\|_{{\rm L}^p(\mu_\kappa)}<\infty.
\end{split}\end{equation*}

Therefore, (1) and (2) together imply that when $(p,q,s)\in[1,\infty)\times[1,\infty]\times(0,\infty)$, if $f\in {\rm L}^p(\mu_\kappa)$, then $\widetilde{{\rm N}}^{\kappa,q}_{s,p}(f)<\infty$ and ${\rm N}^{\kappa,q}_{s,p}(f)<\infty$ are equivalent.
\end{remark}

The next proposition presents further properties of the Besov spaces introduced in Definition \ref{general-besov}.
\begin{proposition}\label{property-besov}
Let $p\in[1,\infty)$, $q\in[1,\infty]$ and $s\in(0,\infty)$. Then, the following properties hold.
\begin{itemize}
\item[(\emph{i})] ${\rm B}_{s,p}^{\kappa,q}(\R^n)$ is a Banach space with respect to the norm
$$\|f\|_{{\rm B}_{s,p}^{\kappa,q}(\R^n)}:=\|f\|_{{\rm L}^p(\mu_\kappa)}+ {\rm N}^{\kappa,q}_{s,p}(f).$$

\item[(\emph{ii})] For $s_1,s_2\in(0,\infty)$ with $s_1\leq s_2$, ${\rm B}_{s_2,p}^{\kappa,q}(\R^n)$ is continuously embedded in ${\rm B}_{s_1,p}^{\kappa,q}(\R^n)$.

\item[(\emph{iii})] Let $f,g\in {\rm B}_{s,p}^{\kappa,q}(\R^n)$. Denote $\Phi=\max\{f,g\}$ and $\Psi=\min\{f,g\}$. Then, $\Phi,\Psi\in{\rm B}_{s,p}^{\kappa,q}(\R^n)$, and
\begin{equation*}
\|\Phi\|_{{\rm B}_{s,p}^{\kappa}(\R^n)}^p+\|\Psi\|_{{\rm B}_{s,p}^{\kappa}(\R^n)}^p\leq \|f\|_{{\rm B}_{s,p}^{\kappa}(\R^n)}^p + \|g\|_{{\rm B}_{s,p}^{\kappa}(\R^n)}^p,\\
\end{equation*}
\begin{equation}\label{rem-A-1}
\|\Phi\|_{{\rm B}_{s,p}^{\kappa,\infty}(\R^n)}^p+\|\Psi\|_{{\rm B}_{s,p}^{\kappa,\infty}(\R^n)}^p\leq \|f\|_{{\rm B}_{s,p}^{\kappa,\infty}(\R^n)}^p + \|g\|_{{\rm B}_{s,p}^{\kappa,\infty}(\R^n)}^p.
\end{equation}
\end{itemize}
\end{proposition}
\begin{proof}
We provide a detailed proof for the inequality \eqref{rem-A-1}. The remaining statements (\emph{i})-(\emph{iii}) in Proposition \ref{property-besov} follow by analogous arguments to those employed in the proof of Proposition 3.6, Lemma 3.7, and Proposition 3.8 of \cite{ZWLL2024+}.

Let $E_1=\{x\in\R^n:\ f\geq g\}$ and $E_2=\{x\in\R^n:\ f<g\}$. We observe that
\begin{equation}\begin{split}\label{pf-rem-A-1}
\|\Phi\|_{{\rm L}^p(\mu_\kappa)}^p + \|\Psi\|_{{\rm L}^p(\mu_\kappa)}^p&=\int_{E_1}|f|^p\,\d\mu_\kappa + \int_{E_2}|g|^p\,\d\mu_\kappa
+ \int_{E_1}|g|^p\,\d\mu_\kappa + \int_{E_2}|f|^p\,\d\mu_\kappa\\
&=\|f\|_{{\rm L}^p(\mu_\kappa)}^p + \|g\|_{{\rm L}^p(\mu_\kappa)}^p.
\end{split}\end{equation}

On the other hand, we decompose that
\begin{equation*}\begin{split}\label{pf-rem-A-2}
&\int_{\R^n} P^\kappa_t(|\Phi-\Phi(x)|^p)(x)\,\mu_\kappa(\d x)\\
&=\int_{E_1}\int_{E_1}p^\kappa_t(x,y)|f(y)-f(x)|^p\,\mu_\kappa(\d y)\mu_\kappa(\d x) \\
&\quad+ \int_{E_2}\int_{E_1}p^\kappa_t(x,y)|f(y)-g(x)|^p\,\mu_\kappa(\d y)\mu_\kappa(\d x)\\
&\quad +\int_{E_1}\int_{E_2}p^\kappa_t(x,y)|g(y)-f(x)|^p\,\mu_\kappa(\d y)\mu_\kappa(\d x)\\
&\quad + \int_{E_2}\int_{E_2}p^\kappa_t(x,y)|g(y)-g(x)|^p\,\mu_\kappa(\d y)\mu_\kappa(\d x)\\
&=: {\rm I}_1 + {\rm I}_2 + {\rm I}_3 + {\rm I}_4,\quad t>0,
\end{split}\end{equation*}
and
\begin{equation*}\begin{split}\label{pf-rem-A-2}
&\int_{\R^n} P^\kappa_t(|\Psi-\Psi(x)|^p)(x)\,\mu_\kappa(\d x)\\
&=\int_{E_1}\int_{E_1}p^\kappa_t(x,y)|g(y)-g(x)|^p\,\mu_\kappa(\d y)\mu_\kappa(\d x) \\
&\quad+ \int_{E_2}\int_{E_1}p^\kappa_t(x,y)|g(y)-f(x)|^p\,\mu_\kappa(\d y)\mu_\kappa(\d x)\\
&\quad +\int_{E_1}\int_{E_2}p^\kappa_t(x,y)|f(y)-g(x)|^p\,\mu_\kappa(\d y)\mu_\kappa(\d x)\\
&\quad + \int_{E_2}\int_{E_2}p^\kappa_t(x,y)|f(y)-f(x)|^p\,\mu_\kappa(\d y)\mu_\kappa(\d x)\\
&=: {\rm J}_1 + {\rm J}_2 + {\rm J}_3 + {\rm J}_4,\quad t>0.
\end{split}\end{equation*}
By applying the rearrangement inequality:
$$|a_0-b_1|^p + |a_1-b_0|^p\leq |a_0-b_0|^p+|a_1-b_1|^p,$$
for all $(a_0,a_1),(b_0,b_1)\in\R^2$ with $(a_0-b_0)(a_1-b_1)\leq0$, we proceed as follows. Let $a_0=g(y)$, $a_1=g(x)$, $b_0=f(y)$ and $b_1=f(x)$, where $x\in E_2$ and $y\in E_1$. The rearrangement inequality yields
\begin{equation}\begin{split}\label{pf-rem-A-3}
{\rm I}_2 + {\rm J}_2 &\leq \int_{E_2}\int_{E_1}p^\kappa_t(x,y)(|g(y)-g(x)|^p+|f(y)-f(x)|^p)\,\mu_\kappa(\d y)\mu_\kappa(\d x), \quad t>0.
\end{split}\end{equation}
Now, take $x\in E_1$ and $y\in E_2$. The same inequality gives
\begin{equation}\begin{split}\label{pf-rem-A-4}
{\rm I}_3 + {\rm J}_3 &\leq \int_{E_2}\int_{E_1}p^\kappa_t(x,y)(|g(y)-g(x)|^p+|f(y)-f(x)|^p)\,\mu_\kappa(\d y)\mu_\kappa(\d x), \quad t>0.
\end{split}\end{equation}
Combining the bounds from \eqref{pf-rem-A-3} and \eqref{pf-rem-A-4} with the remaining terms ${\rm I}_1, {\rm I}_4, {\rm J}_1,{\rm J}_4$, we arrive at
\begin{equation*}\begin{split}\label{pf-rem-A-5}
&\int_{\R^n} P^\kappa_t(|\Phi-\Phi(x)|^p)(x)\,\mu_\kappa(\d x) + \int_{\R^n} P^\kappa_t(|\Psi-\Psi(x)|^p)(x)\,\mu_\kappa(\d x)\\
&\leq \int_{\R^n} P^\kappa_t(|f-f(x)|^p)(x)\,\mu_\kappa(\d x) + \int_{\R^n} P^\kappa_t(|g-g(x)|^p)(x)\,\mu_\kappa(\d x),\quad t>0.
\end{split}\end{equation*}
Multiplying on both sides of this inequality by $t^{-\frac{ps}{2}}$, we obtain
\begin{equation}\begin{split}\label{pf-rem-A-6}
&t^{-\frac{ps}{2}}\int_{\R^n} P^\kappa_t(|\Phi-\Phi(x)|^p)(x)\,\mu_\kappa(\d x) + t^{-\frac{ps}{2}}\int_{\R^n} P^\kappa_t(|\Psi-\Psi(x)|^p)(x)\,\mu_\kappa(\d x)\\
&\leq {\rm N}^{\kappa,\infty}_{s,p}(f)^p + {\rm N}^{\kappa,\infty}_{s,p}(g)^p<\infty,\quad t>0.
\end{split}\end{equation}
Since $\Phi, \Psi \in {\rm B}_{s,p}^{\kappa,q}(\R^n)$, Remark \ref{rem-A}(2) allows us to take the limsup in \eqref{pf-rem-A-6} as $t \to 0^+$ and conclude
\begin{equation}\begin{split}\label{pf-rem-A-7}
&{\rm N}^{\kappa,\infty}_{s,p}(\Phi)^p + {\rm N}^{\kappa,\infty}_{s,p}(\Psi)^p\\
&=\limsup_{t\rightarrow0^+}t^{-\frac{ps}{2}}\int_{\R^n} P^\kappa_t(|\Phi-\Phi(x)|^p)(x)\,\mu_\kappa(\d x)\\
 &\quad + \limsup_{t\rightarrow0^+}t^{-\frac{ps}{2}}\int_{\R^n} P^\kappa_t(|\Psi-\Psi(x)|^p)(x)\,\mu_\kappa(\d x)\\
&=\limsup_{t\rightarrow0^+}t^{-\frac{ps}{2}}\bigg\{\int_{\R^n} P^\kappa_t(|\Phi-\Phi(x)|^p)(x)\,\mu_\kappa(\d x) + \int_{\R^n} P^\kappa_t(|\Psi-\Psi(x)|^p)(x)\,\mu_\kappa(\d x)\bigg\}\\
&\leq {\rm N}^{\kappa,\infty}_{s,p}(f)^p + {\rm N}^{\kappa,\infty}_{s,p}(g)^p.
\end{split}\end{equation}

Thus, combining \eqref{pf-rem-A-1} with \eqref{pf-rem-A-7} completes the proof of  \eqref{rem-A-1}.

\end{proof}

\section{Appendix}\label{app-B}\hskip\parindent
In this part, we provide some elementary properties on the $s$-D-perimeter given in Definition \ref{rel-s-per}.
\begin{proposition}\label{property-D-perimeter}
Let $\Omega\subset\R^n$ be an open set and $s\in(0,1/2)$. Then, the following properties hold.
\begin{itemize}
\item[(\emph{1})](\emph{$G$-invariance})  For every measurable subsets $A\subset\R^n$ and each $g\in G$,
$${\rm Per}^\kappa_s(gA,g\Omega)={\rm Per}^\kappa_s(A,\Omega).$$

\item[(\emph{2})](\emph{Subadditivity}) For any measurable subsets $A,B\subset\R^n$, the following subadditivity holds:
$${\rm Per}^\kappa_s (A\cup B,\Omega)\leq{\rm Per}^\kappa_s (A,\Omega)+{\rm Per}^\kappa_s (B,\Omega).$$

\item[(\emph{3})](\emph{Monotonicity in the domain}) Let $U_1,U_2\subset\R^n$ be measurable open set with $U_1\subset U_2$. Then, for any measurable set $A\subset\R^n$,
$${\rm Per}^\kappa_s (A,U_1)\leq {\rm Per}^\kappa_s (A,U_2).$$

\item[(\emph{4})](\emph{Non-monotonicity in the set}) There exist measurable sets $A,B\subset\R^n$ with $A\subset B$ such that
$${\rm Per}^\kappa_s(A,\Omega)>{\rm Per}^\kappa_s(B,\Omega).$$
In particular, the functional ${\rm Per}^\kappa_s(\cdot, \Omega)$ need not be increasing with respect to set inclusion.
\end{itemize}
\end{proposition}
\begin{proof} Property (\emph{2}) follows from the same argument as in  \cite[Proposition
2.1]{DFPV}, applied to the definition of ${\rm Per}^\kappa_s$. By the definition of ${\rm Per}^\kappa_s$,  a  direct computation leads to
\begin{equation*}\begin{split}
{\rm Per}^\kappa_s(A,U_2)&={\rm Per}^\kappa_s(A,U_1)+  2L^\kappa_s(A\cap U_1^c\cap U_2, E^c\cap U_1^c)\\
&\quad+ 2L^\kappa_s(A^c\cap U_1^c\cap U_2, A\cap U_2^c),
\end{split}\end{equation*}
which clearly implies (\emph{3}). To derive (\emph{4}), see the proof of \cite[Proposition 2.3]{DFPV} in the particular case when $\kappa\equiv0$. In what follows, we turn to prove (\emph{1}).

The Dunkl heat kernel admits the explicit representation (see, e.g., \cite[Section 4]{Rosler1998}):
$$p^\kappa_t(x,y)=\frac{1}{\mathfrak{c}_\kappa(2t)^{n+2\chi}}\exp\Big(\frac{|x|^2+|y|^2}{4t}\Big)E_\kappa\Big(\frac{x}{\sqrt{2t}}, \frac{y}{\sqrt{2t}}\Big),\quad x,y\in\R^n,\,t>0,$$
where $\mathfrak{c}_\kappa$ is the Macdonald--Mehta constant (defined in Section \ref{sec-intro}), and $E_\kappa(\cdot,\cdot)$ is the Dunkl kernel (initially introduced in \cite{Dunkl1991}) associated with the Dunkl operator $T_\kappa^\xi$. It is known that $E_\kappa(\cdot,\cdot)$ can be uniquely extended to a holomorphic function in $\mathbb{C}^d\times\mathbb{C}^d$, and it is $G$-invariant, i.e.,  $E_\kappa(gx,gy)=E_\kappa(x,y)$ for all $x,y\in\R^n$ and $g\in G$ (see \cite[Section 2.5]{Rosler2003} for details and more properties on the Dunkl kernel), where $\mathbb{C}$ denotes the set of complex numbers. This immediately implies the $G$-invariance of $p^\kappa_t(\cdot,\cdot)$. Combining this with the $G$-invariance of the measure $\mu_\kappa$, we obtain for any measurable sets $E,F\subset\R^n$,
\begin{equation*}\begin{split}\label{pf-app-B-1}
L^\kappa_s(gE,gF)&=\int_0^\infty t^{-(1+s)}\int_{gE}\int_{gF}p_t^\kappa(x,y)\,\mu_\kappa(\d y)\mu_\kappa(\d x)\d t\\
&=\int_0^\infty t^{-(1+s)}\int_{E}\int_{F}p_t^\kappa(gx,gy)\,\mu_\kappa(\d y)\mu_\kappa(\d x)\d t\\
&=L^\kappa_s(E,F).
\end{split}\end{equation*}
The conclusion follows by observing that $g E^c=(g E)^c$ and $gE \cap g F=g(E\cap F)$ for any $g\in G$. Substituting these into the definition of ${\rm Per}^\kappa_s$ finishes the proof of  (\emph{1}).
\end{proof}

Let $\Omega$ and $s$ be as in Proposition \ref{property-D-perimeter}, and let $A\subset\R^n$ be measurable.  In the particular case where $\kappa\equiv0$, it was proved in \cite[Proposition 3.12]{Lom2019} that the following geometric properties hold:
\begin{itemize}
\item[(\emph{a})](Scaling invariance) For any $r>0$, ${\rm Per}^0_s(rA,r\Omega)=r^{n-s}{\rm Per}^0_s(A,\Omega)$.

\item[(\emph{b})](Translation invariance) For any $z\in\R^n$, ${\rm Per}^0_s(A+z,\Omega+z)={\rm Per}^0_s(A,\Omega)$.
\end{itemize}
However, in general, the functional ${\rm Per}^\kappa_s$ typically fail to satisfy these invariance properties, due to that the measure $\mu_\kappa$ is not translation-invariant, breaking property (\emph{b}), and the Dunkl heat kernel $(p_t^\kappa)_{t>0}$ lacks the homogeneous scaling behavior required to preserve (\emph{a}).

\end{appendix}


\begin{thebibliography}{a23}
\bibitem{AmDePhMa2011}
L. Ambrosio, G. De Philippis, L. Martinazzi: Gamma-convergence of nonlocal perimeter functionals. Manuscripta Math. 134 (2011), no. 3-4,  377--403.

\bibitem{Anker2017}
J.-P. Anker:  An introduction to Dunkl theory and its analytic aspects. \emph{Analytic, algebraic and geometric aspects of differential equations}, 3--58, Trends Math., Birkh\"{a}user/Springer, Cham, 2017.

\bibitem{ADH2019}
J.-P. Anker, J. Dziuba\'{n}ski, A. Hejna: Harmonic functions, conjugate harmonic functions and the Hardy space $H^1$ in the rational Dunkl setting. J. Fourier Anal. Appl. 25
(2019), 2356--2418.

\bibitem{BP2019}
J. Berendsen, V. Pagliari: On the asymptotic behaviour of nonlocal perimeters. ESAIM Control Optim. Calc. Var. 25 (2019), Paper No. 48, 27 pp.

\bibitem{BisPer2017}
C.J. Bishop, Y. Peres: \emph{Fractals in probability and analysis}. Cambridge Studies in Advanced Mathematics, 162. Cambridge University Press, Cambridge, 2017.

\bibitem{BBM1}
J. Bourgain, H. Brezis, P. Mironescu: Another look at Sobolev spaces. In: \emph{Optimal control and partial differential equations}, IOS, Amsterdam 2001, 439--455.

\bibitem{BBM2}
J. Bourgain, H. Brezis, P. Mironescu: Limiting embedding theorems for $W^{s,p}$ when $s\nearrow1$ and applications. Dedicated to the memory of Thomas H. Wolff. J. Anal. Math.
87 (2002), 77--101.

\bibitem{BGT2022}
F. Buseghin, N. Garofalo, G. Tralli: On the limiting behaviour of some nonlocal seminorms: a new phenomenon. Ann. Sc. Norm. Super. Pisa Cl. Sci. (5) Vol. XXIII (2022), 837--875.

\bibitem{CRS2009}
L. Caffarelli, J.-M. Roquejoffre, O. Savin: Nonlocal minimal surfaces. Comm. Pure Appl. Math. 63 (2010), 1111--1144.

\bibitem{CaffValdV2011}
L. Caffarelli, E. Valdinoci: Uniform estimates and limiting arguments for nonlocal minimal surfaces. Calc. Var. Partial Differential Equations 41 (2011), no. 1-2, 203--240.

\bibitem{CCLMP}
A. Carbotti, S. Cito, D.A. La Manna, D. Pallara: Asymptotics of the $s$-fractional Gaussian perimeter as $s\rightarrow0^+$.  Fract. Calc. Appl. Anal. 25 (2022), no. 4,
1388--1403.

\bibitem{CG2024}
M. Caselli, L. Gennaioli: Asymptotics as $s\rightarrow0^+$ of the fractional perimeter on Riemannian manifolds. Preprint (2023), arXiv:2306.11590v5.

\bibitem{CDFB2023}
I. Ceresa Dussel, J. Fern\'{a}ndez Bonder: A Bourgain-Brezis-Mironescu formula for anisotropic fractional Sobolev spaces and applications to anisotropic fractional
differential equations. J. Math. Anal. Appl. 519 (2023), no. 2, Paper No. 126805, 25 pp.

\bibitem{Cherednik1991}
I. Cherednik: A unification of Knizhnik-Zamolodchikov and Dunkl operators via affine Hecke algebras. Invent. Math. 106 (1991), no. 2, 411--431.

\bibitem{Cherednik1994}
I. Cherednik: Integration of quantum many-body problems by affine Knizhnik-Zamolodchikov equations. Adv. Math. 106 (1994), no. 1, 65--95.

\bibitem{CDLKNP2023}
V. Crismale, L. De Luca, A. Kubin, A. Ninno, M. Ponsiglione: The variational approach to $s$-fractional heat flows and the limit cases $s\rightarrow0^+$ and $s\rightarrow1^-$. J. Funct. Anal. 284 (2023), no. 8, Paper No. 109851, 38 pp.

\bibitem{DaiFeng2016}
F. Dai, H. Feng: Riesz transforms and fractional integration for orthogonal expansions on spheres, balls and simplices. Adv. Math. 301 (2016), 549--614.

\bibitem{DaiXu2013}
F. Dai, Y. Xu:  \emph{Approximation theory and harmonic analysis on spheres and balls}. Springer Monographs in Mathematics. Springer, New York, 2013.

\bibitem{Davies1989}
E.B. Davies: \emph{Heat Kernels and Spectral Theory}. Cambridge Tracts in Mathematics, 92. Cambridge University Press, Cambridge, 1989.

\bibitem{Davila02}
J. D\'{a}vila: On an open question about functions of bounded variation. Calc. Var. Partial Differential Equations 15 (2002), 519--527.

\bibitem{Jeu}
M.F.E. de Jeu: The Dunkl transform. Invent. Math. 113 (1993), 147--162.

\bibitem{DNPV2012}
E. Di Nezza, G. Palatucci, E. Valdinoci: Hitchhiker's guide to the fractional Sobolev spaces. Bull. Sci. Math. 136 (2012), no. 5, 521--573.

\bibitem{DFPV}
S. Dipierro, A. Figalli, G. Palatucci, E. Valdinoci: Asymptotics of the $s$-perimeter as $s\searrow0$. Discrete Contin. Dyn. Syst. 33 (2013), no. 7, 2777--2790.

\bibitem{DLTYY2024}
O. Dom\'{\i}nguez, Y. Li, S. Tikhonov, D. Yang, W. Yuan: A unified approach to self-improving property via $K$-functionals. Calc. Var. Partial Differential Equations 63 (2024), no. 9, Paper No. 231.

\bibitem{Dunkl1988}
C.F. Dunkl: Reflection groups and orthogonal polynomials on the sphere. Math. Z. 197 (1988), no. 1, 33--60.

\bibitem{Dunkl1989}
C.F. Dunkl: Differential-difference operators associated to reflection groups. Trans. Amer. Math. Soc. 311 (1989), 167--183.

\bibitem{Dunkl1991}
C.F. Dunkl: Integral kernels with reflection group invariance. Canad. J. Math. 43 (1991), no. 6, 1213--1227.

\bibitem{Dunkl1992}
C.F. Dunkl: Hankel transforms associated to finite reflection groups. \emph{Hypergeometric functions on domains of positivity, Jack polynomials, and applications} (Tampa, FL, 1991), 123--138, Contemp. Math., 138, Amer. Math. Soc., Providence, RI, 1992.

\bibitem{DunklXu2014}
C.F. Dunkl, Y. Xu, Y: \emph{Orthogonal polynomials of several variables}. Second edition. Encyclopedia of Mathematics and its Applications 155, Cambridge University Press,
Cambridge, 2014.

\bibitem{EvaGar2015}
L.C. Evans, R.F. Gariepy: \emph{Measure theory and fine properties of functions}. Revised edition. Textbooks in Mathematics. CRC Press, Boca Raton, FL, 2015.

\bibitem{Etingof}
P. Etingof: A uniform proof of the Macdonald-Mehta-Opdam identity for finite Coxeter groups. Math. Res. Lett. 17 (2010), no. 2, 275--282.

\bibitem{Falconer2014}
K. Falconer: \emph{Fractal geometry: Mathematical foundations and applications}. Third edition. John Wiley \& Sons, Ltd., Chichester, 2014.

\bibitem{GalYor2006}
L. Gallardo, M. Yor: A chaotic representation property of the multidimensional Dunkl processes. Ann. Probab. 34 (2006), no. 4, 1530--1549.

\bibitem{GalYor2005}
L. Gallardo, M. Yor: Some new examples of Markov processes which enjoy the time-inversion property, Probab. Theory Related Fields 132 (2005), 150--162.

\bibitem{GT2020a}
N. Garofalo, G. Tralli: Functional inequalities for class of nonlocal hypoelliptic equations of H\"{o}rmander type. Nonlinear Anal. 193 (2020), Paper No. 111567, 23 pp.

\bibitem{GT2020b}
N. Garofalo, G. Tralli: Nonlocal isoperimetric inequalities for Kolmogorov-Fokker-Planck operators. J. Funct. Anal. 279 (2020), no. 3, Paper No. 108591, 40 pp.

\bibitem{GT2024}
N. Garofalo, G. Tralli: A universal heat semigroup characterisation of Sobolev and BV spaces in Carnot groups. Int. Math. Res. Not. IMRN no. 8, (2024), 6731--6758.

\bibitem{Grafakos2014}
L. Grafakos: \emph{Classical Fourier analysis}. Third edition. Graduate Texts in Mathematics, 249. Springer, New York, 2014.

\bibitem{HPXZ}
B.-X. Han, A. Pinamonti, Z. Xu, K. Zambanini: Maz'ya--Shaposhnikova meet Bishop--Gromov. Preprint (2024), arXiv:2402.11174.

\bibitem{Hec1987}
G.J. Heckman: Root systems and hypergeometric functions. II. Compositio Math. 64 (1987), no. 3, 353--373.

\bibitem{HecOpd1987}
G.J. Heckman, E.M. Opdam: Root systems and hypergeometric functions. I. Compositio Math. 64 (1987), no. 3, 329--352.

\bibitem{Jameson}
G.J. O. Jameson: The incomplete gamma functions. Math. Gaz. 100 (2016), no. 548, 298--306.

\bibitem{KMPY84}
J.L. Kaplan, J. Mallet-Paret, J.A. Yorke: The Lyapunov dimension of a nowhere differentiable attracting torus. Ergodic Theory Dynam. Systems 4 (1984), no. 2, 261--281.

\bibitem{KMX2005}
G.E. Karadzhov, M. Milman, J. Xiao: Limits of higher-order Besov spaces and sharp reiteration theorems. J. Funct. Anal. 221 (2005), 323--339.

\bibitem{Leoni2023}
G. Leoni:  \emph{A first course in fractional Sobolev spaces}. Graduate Studies in Mathematics, 229. American Mathematical Society, Providence, RI, 2023.

\bibitem{Li2019}
H. Li: Weak type estimates for square functions of Dunkl heat flows. Acta Math. Appl. Sin. Engl. Ser. (2025). https://doi.org/10.1007/s10255-025-0029-2.

\bibitem{LiWang2024}
H. Li, K. Wang: Dimension-free Maz'ya--Shaposhnikova limiting formulas in Grushin spaces. Atti Accad. Naz. Lincei Rend. Lincei Mat. Appl. 35 (2024), no. 4, 627--642.

\bibitem{LZ2023}
H. Li, M. Zhao: Dimension-free square function estimates for Dunkl operators. Math. Nachr. 296  (2023), 1225--1243.

\bibitem{Lom2019}
L. Lombardini: Fractional perimeters from a fractal perspective. Adv. Nonlinear Stud. 19 (2019), no. 1, 165--196.

\bibitem{Lud2014}
M. Ludwig: Anisotropic fractional perimeters. J. Differential Geom. 96 (2014), no. 1, 77--93.

\bibitem{Ludwig2014}
M. Ludwig: Anisotropic fractional Sobolev norms. Adv. Math. 252 (2014), 150--157.

\bibitem{MRT2019}
J.M. Maz\'{o}n, J.D. Rossi, J. Toledo: \emph{Nonlocal perimeter, curvature and minimal surfaces for measurable sets}. Frontiers in Mathematics. Birkh\"{a}user/Springer, Cham, 2019.

\bibitem{MS2002}
V. Maz'ya, T. Shaposhnikova: On the Bourgain, Brezis, and Mironescu theorem concerning limiting embeddings of fractional Sobolev spaces. J. Funct. Anal. 195 (2002), 230--238.
(Erratum-ibid. 201 (2003), 298--300.)


\bibitem{Opdam88b}
E.M. Opdam: Root systems and hypergeometric functions. IV. Compositio Math. 67 (1988), no. 2, 191--209.

\bibitem{Opdam88a}
E.M. Opdam: Root systems and hypergeometric functions. III.  Compositio Math. 67 (1988), no. 1, 21--49.


\bibitem{Opdam89}
E.M. Opdam: Some applications of hypergeometric shift operators. Invent. Math. 98 (1989), 1--18.

\bibitem{PSV2017}
A. Pinamonti, M. Squassina, E. Vecchi: The Maz'ya--Shaposhnikova limit in the magnetic setting. J. Math. Anal. Appl. 449 (2017), no. 2, 1152--1159.


\bibitem{RenShen2021}
H. Ren, W. Shen: A dichotomy for the Weierstrass-type functions. Invent. Math. 226 (2021), no. 3, 1057--1100.

\bibitem{Rosler1998}
M. R\"{o}sler: Generalized Hermite polynomials and the heat equation for Dunkl operators. Comm. Math. Phys. 192 (1998), no. 3, 519--542.

\bibitem{Rosler1999}
M. R\"{o}sler: Positivity of Dunkl's intertwining operator.  Duke Math. J. 98 (1999), no. 3, 445--463.

\bibitem{Rosler03}
M. R\"{o}sler: A positive radial product formula for the Dunkl kernel. Trans. Amer. Math. Soc. 355 (2003), no. 6, 2413--2438.

\bibitem{Rosler2003}
M. R\"{o}sler: Dunkl operators: theory and applications. \emph{Orthogonal polynomials and special functions (Leuven, 2002)}, 93--135, Lecture Notes in Math., 1817, Springer, Berlin, 2003.

\bibitem{RoslerVoit1998}
M. R\"{o}sler, M. Voit: Markov processes related with Dunkl operators. Adv. in Appl. Math. 21 (1998), no. 4, 575--643.

\bibitem{Serra2024}
J. Serra: Nonlocal minimal surfaces: recent developments, applications, and future directions. SeMA J. 81 (2024), no. 2, 165--191.

\bibitem{Shen2018}
W. Shen: Hausdorff dimension of the graphs of the classical Weierstrass functions. Math. Z. 289 (2018), no. 1-2, 223--266.

\bibitem{Taibleson1964}
M.H. Taibleson: On the theory of Lipschitz spaces of distributions on Euclidean n-space. I. Principal properties. J. Math. Mech. 13 (1964), 407--479.

\bibitem{Taibleson1965}
M.H. Taibleson: On the theory of Lipschitz spaces of distributions on Euclidean n-space. II. Translation invariant operators, duality, and interpolation. J. Math. Mech. 14 (1965), 821--839.

\bibitem{ThaXu05}
S. Thangavelu, Y. Xu: Convolution operator and maximal function for the Dunkl transform. J. Anal. Math. 97 (2005), 25--55.

\bibitem{ThaXu07}
S. Thangavelu, Y. Xu: Riesz transform and Riesz potentials for Dunkl transform. J. Comput. Appl. Math. 199 (2007), no. 1, 181--195.

\bibitem{Vel2020}
A. Velicu: Sobolev-type inequalities for Dunkl operators. J. Funct.  Analysis 279 (2020), Paper No. 108695, 37 pp.

\bibitem{Vis1991}
A. Visintin: Generalized coarea formula and fractal sets. Japan J. Indust. Appl. Math. 8 (1991), no. 2, 175--201.

\bibitem{ZWLL2024+}
N. Zhao, Z. Wang, P. Li, Y. Liu: Geometric topics related to Besov type spaces on the Grushin setting. Potential Anal. (2024),
https://doi.org/10.1007/s11118-024-10187-9.

\bibitem{Z2005}
D. \v{Z}ubrini\'{c}: Analysis of Minkowski contents of fractal sets and applications. Real Anal. Exchange 31 (2005/06), no. 2, 315--354.

\end{thebibliography}
\end{document}